\documentclass[11pt,reqno]{amsart}
\usepackage[utf8]{inputenc}

\usepackage[utf8]{inputenc}
\usepackage{amsfonts}
\usepackage{amsthm}
\usepackage{amsmath}
\usepackage{amssymb}
\usepackage{stmaryrd}
\usepackage{yfonts}
\usepackage{MnSymbol}
\usepackage{mathtools}
\usepackage{fullpage}
\usepackage{color}
\usepackage{enumitem}
\usepackage{tikz}
\usepackage{capt-of}
\usepackage{tikzscale}

\date{\today}

\newtheorem{teo}{Theorem}[section]

\newtheorem{coro}[teo]{Corollary}
\newtheorem{lema}[teo]{Lemma}

\newtheorem{prop}[teo]{Proposition}
\newtheorem{defi}[teo]{Definition}
\newtheorem{ex}[teo]{Example}
\newtheorem{question}[teo]{Question}

\DeclareMathOperator{\C}{\mathcal{C}}
\DeclareMathOperator{\B}{\mathcal{B}}
\DeclareMathOperator{\N}{\mathbb{N}}

\DeclareMathOperator{\Q}{\mathbb{Q}}

\DeclareMathOperator{\fin}{\sf FIN}

\newcommand{\cantor}{2^{\N}}
\newcommand{\ideal}{\mathcal{I}}
\newcommand{\idealj}{\mathcal{J}}

\def\exh{\mbox{\sf Exh}}
\def\suma{\mbox{\sf Sum}}
\def\su{\subseteq}

\title{Pathology of submeasures and $F_{\sigma}$ ideals}
\author{Jorge Mart\'inez}
\address{Escuela de Matem\'aticas, Universidad Industrial de Santander, C.P. 680001, Bucaramanga - Colombia.}
\email{sebas-123126@hotmail.com}

\author{David Meza-Alc\'antara}
\address{Facultad de Ciencias, Universidad Nacional Autónoma de México,  CDMX}
\email{dmeza@ciencias.unam.mx}

\author{Carlos Uzc\'{a}tegui}
\address{Escuela de Matem\'aticas, Universidad Industrial de Santander, C.P. 680001, Bucaramanga - Colombia.}
\email{cuzcatea@saber.uis.edu.co}

\date{\today}

\begin{document}

\begin{abstract}
We address some phenomena about the interaction between lower semicontinuous submeasures on $\mathbb{N}$ and $F_{\sigma}$ ideals.   We analyze the pathology degree of a submeasure  and  present a method to construct pathological $F_{\sigma}$ ideals. We give a partial answers to the question of whether every   nonpathological tall $F_{\sigma}$ ideal is Kat\v{e}tov above the random ideal or at least has a Borel selector. Finally, we show  a representation of nonpathological $F_{\sigma}$ ideals using sequences in Banach spaces. 
\end{abstract}

\thanks{This work was supported by UNAM-PAPIIT IN119320. Second author was supported by a UNAM-DGAPA-PASPA grant for a sabbatical stay at York University, Toronto, Canada.}

\maketitle

\section{Introduction}

Mazur \cite{Mazur91} showed that every 
$F_{\sigma}$ ideal on $\N$ is of the form 
$\fin(\varphi)=\{A\su \N:\varphi(A)<\infty\}$
for some lower semicontinuous submeasure (lscsm) $\varphi$ on $\N$.
Mazur's proof may be thought as a canonical way to get a submeasure for an $F_{\sigma}$ ideal, given a representation as a countable union of closed sets. Mazur's construction provides integer-valued submeasures, but it hides interesting properties of both, the ideal and the submeasure. For example, it does not show whether the ideal is tall or how close is the submeasure to be a measure, even in the case of summable ideals, which are ideals induced by measures. 

Farah \cite{Farah2000} introduced the degree of pathology to quantify the closeness of a lscsm to be a measure. A lscsm  has  degree 1, and is called {\em nonpathological}, when is the supremum of a family of measures. We say that an $F_{\sigma}$ ideal $\ideal$ is {\em nonpathological} if $\ideal=\fin(\varphi)$ for some nonpathological lscsm $\varphi$.

We study pathological and nonpathological $F_{\sigma}$ ideals from several points of view. 
Pathological $F_{\sigma}$ ideals are those that whenever $\ideal=\fin(\varphi)$ for some lscsm $\varphi$, the degree of pathology of $\varphi$ is infinite.
We show a family of examples of this kind of ideals, based on an ideal defined by K. Mazur in \cite{Mazur91}, and also present a way to identify some lscsm $\varphi$ whose degree of pathology is infinite. This last condition is necessary for $\fin(\varphi)$ to be pathological but it is unknown whether is sufficient. 
We use this criterion to show that the degree of pathology of the usual lscsm $\chi$ inducing the Solecki's ideal $\mathcal{S}$ is equal to infinite. 
That complements Figueroa and Hru\v{s}\'ak's result showing that $\mathcal{S}$ is pathological \cite{FigueroaHrusak2023}. 
We also prove that nonpathology is preserved downward by the  Rudin-Keisler pre-order. 
We show that our example of a pathological ideal has a restriction which is Rudin-Keisler above the Solecki's ideal. 
Some questions about pathology of submeasures, pathology of ideals and Rudin-Keisler and Kat\v{e}tov orders are stated. 

The second aspect of our study is concerned with the class of tall ideals (those satisfying that for every infinite set $A$ there is an infinite set  $B\subseteq A$ in the ideal).
Tall ideals have been extensively  investigated (see for instance  \cite{Hrusak2011,HMTU2017,uzcasurvey}).  A very useful tool for the study of tall ideals is the Kat\v{e}tov pre-order $\leq_K$. Among tall Borel ideals  one which has played a pivotal role is the random ideal $\mathcal{R}$ generated by the cliques and independent sets of the random graph \cite{Hrusak2017}. It is known that if $\mathcal{R}\leq_K \ideal$, then $\ideal$ is tall. For a while it was conjectured that $\mathcal{R}$ was a $\leq_K$-minimum among Borel tall ideals \cite{HMTU2017}, this  turned out to be false \cite{GrebikHrusak2020}, in fact, there are $F_\sigma$ tall ideals which are not $\leq_K$ above $\mathcal{R}$. 
We show that if $\varphi$ is a nonpathological lscsm of type $c_0$ and $\fin(\varphi)$ is tall, then  $\mathcal{R}\leq_K\fin(\varphi)$. The notion of a lscsm of type $c_0$ was motivated by the results presented in the last section. It is known that $\mathcal{R}\leq_K \ideal$ implies that $\ideal$   has a Borel selector (i.e. the set $B$ in the definition of tallness can be found in a Borel way from $A$) \cite{GrebikHrusak2020,GrebikUzca2018,grebik2020tall}.  We give a partial answer to the  question of whether every nonpathological  $F_\sigma$ tall ideal  has a Borel selector.

Finally,  in the last section, following the ideas introduced in  \cite{Borodulinetal2015,Drewnowski}, we show how to represent $F_{\sigma}$ ideals using sequences in a Banach space.  Let  ${\bf x}=(x_n)_n$ be a sequence in  a Banach space $X$. We say that  $\sum x_n$ is perfectly bounded, if there is  $k>0$ such that for all $F\subset \N$ finite, $
\left \| \sum_{n\in F}x_n \right \|\leq k$.
Let $$
\mathcal{B}({\bf x})=\left \{ A\subseteq \mathbb{N}: \sum_{n\in A} x_n
 \text{ is perfectly bounded }\right \}.
 $$
We will show that an $F_{\sigma}$ ideal  is nonpathological if and only if it is of the form $\mathcal{B}({\bf x})$ for some ${\bf x}$. In particular, when the space is $c_0$ we get the notion of a lscsm of type $c_0$ mentioned before and we show that $\B({\bf x})$ is tall iff $(x_n)_n$ is weakly null.

\section{Preliminaries}
\label{preliminares}
For an arbitrary set $X$ and a cardinal number $\kappa$, we denote by $[X]^{\kappa}$ (respectively $[X]^{<\kappa}$) the set of the subsets of $X$ having cardinality $\kappa$ (resp, having cardinality smaller than $\kappa$). 
We say that a collection $\mathcal{A}$ of subsets of  a countable set $X$ is {\em analytic} (resp. Borel),
if $\mathcal{A}$ is analytic  (resp. Borel) as a subset of the Cantor cube $2^X$
(identifying subsets of $X$ with characteristic functions). 
The collections $[X]^\omega$, $[X]^{<\omega}$ and $[X]^k$ ($k\in\omega$) are endowed with the subspace topology as subsets of $2^X$.
We refer the reader to  \cite{Kechris94} for all non explained descriptive set theoretic notions and notations.  

An ideal  $\ideal$  on a set $X$ is a collection of subsets of $X$ such that (i) $\emptyset \in \ideal$ and $X\nin \ideal$, (ii) If $A, B\in \ideal$, then $A\cup B\in \ideal$ and (iii) If $A\su B$ and $B\in \ideal$, then $A\in \ideal$. 
Given an ideal $\ideal$ on $X$, the {\em dual filter} of $\ideal$, denoted $\ideal^*$,  is the collection of all sets $X\setminus A$ with $A\in \ideal$.  We denote by  $\ideal^+$ the collection of all subsets of  $X$  which do not belong to $\ideal$. 
The ideal of all finite subsets of $\N$ is denoted by $\fin$.  
We write $A\su^*B$ if $A\setminus B$ is finite. 
There is a vast literature about ideals on countable sets (see for instance the surveys \cite{Hrusak2011} and \cite{uzcasurvey}).  
Since the collection of finite subsets of $X$ is a dense set in $2^X$, there are no ideals containing $[X]^{<\omega}$ which are closed as subsets of $2^X$.  An ideal $\ideal$ on $\N$ is $F_\sigma$ if there is a countable collection of closed subsets $\mathcal{K}_n\su 2^X$ such that $\ideal=\bigcup_n \mathcal{K}_n$. On the other hand,  there are no  $G_\delta$ ideals containing  all finite sets. Thus  the simplest Borel ideals (containing all finite sets)  have complexity $F_\sigma$. 

A family  $\mathcal{A}$ (not necessarily an ideal) of subsets of $X$ is {\em tall}, if every infinite subset of $X$ contains an infinite subset that belongs to $\mathcal{A}$. 
A tall family  $\mathcal{A}$ admits a {\em Borel selector}, if there is a Borel function $S: [X]^\omega\to  [X]^\omega$  such that $S(E)\su E$ and  $S(E)\in \mathcal{A}$ for all $E$.   

A coloring is a function  $c:[X]^2\to 2$ , where $[X]^2$ is the collection of two elements subsets of $X$. A set $H\su X$ is {\em $c$-homogeneous}, if $c$ is constant in $[H]^2$.  We denote by $\hom(c)$ the collection of homogeneous sets and by $Hom(c)$  the ideal generated by the $c$-homogeneous sets, that is, $A\in Hom(c)$ iff there are $c$-homogeneous sets $H_1, \cdots, H_n$ such that $
A\subseteq H_1\cup\cdots \cup H_n$. Since the singletons are trivially $c$-homogeneous sets, $[X]^{<\omega}\subseteq Hom(c)$, for every coloring $c$.
It is easy to check that $Hom(c)$ is $F_\sigma$ and, by Ramsey's theorem,  $Hom(c)$ is tall.
For some colorings $c$,  $Hom(c)$ is trivial. For example, if $c$ satisfies that there are no infinitely many maximal 0-homogeneous sets, say $H_1, \cdots, H_n$ are the maximal 0-homogeneous sets, then $Hom(c)$ is trivial. In fact,  let $x,y\nin H_1\cup\cdots\cup H_n$. Then necessarily $c\{x,y\}=1$ (otherwise there is $i$ such that $x,y\in H_i$). That is, $L=\N\setminus (H_1\cup\cdots\cup H_n)$ is 1-homogeneous. Hence $\N$ is the union of finitely many homogeneous sets. 
The collection of homogeneous sets is a typical example of a tall family that has a Borel selector (see \cite{GrebikUzca2018}). 

A function  $\varphi:\mathcal{P}(\N)\to[0,\infty]$ is a {\em lower semicontinuous submeasure (lscsm)} if $\varphi(\emptyset)=0$,
$\varphi(A)\leq\varphi(A\cup B)\leq\varphi(A)+\varphi(B)$, $\varphi(\{n\})<\infty$ for all $n$ and $\varphi(A)=\lim_{n\to\infty}\varphi(A\cap\{0,1,\cdots, n\})$. 

Three  ideals associated to a lscsm are the following:
\[
\begin{array}{lcl}
\fin(\varphi)& =& \{A\subseteq\N:\varphi(A)<\infty\}.
\\
\exh(\varphi)& = &\{A\subseteq\N: \lim_{n\to\infty}\varphi(A\setminus\{0,1,\dots, n\})=0\}.\\
\suma(\varphi) & = & \{A\subseteq \mathbb{N}:\sum_{n\in A}\varphi(\{n\})<\infty\}. 
\end{array}
\]
Notice that $ \suma(\varphi)\subseteq \exh(\varphi)\subseteq \fin(\varphi)$. 
These ideals have been extensively investigated. The work of Farah \cite{Farah2000} and Solecki \cite{Solecki1999} are two of the most important early works for the study of the ideals associated to submeasures.   

An ideal $\ideal$ is a {\em $P$-ideal} if for every sequence $(A_n)_n$ of sets in $\ideal$ there is $A\in \ideal$ such that 
$A_n\setminus A$ is finite for all $n$. The following representation  of analytic $P$-ideals is the most fundamental result about them.  It says that any $P$-ideal is in a sense similar to a density ideal. 

\begin{teo}
\label{solecki}
(S. Solecki \cite{Solecki1999}) Let $\ideal$ be an analytic ideal on $\N$. The following are equivalent: 
\begin{itemize}
\item[(i)]  $\ideal$ is a $P$-ideal.
\item[(ii)] There is a lscsm $\varphi$ such that $\ideal=\exh(\varphi)$. Moreover,  there is such $\varphi$  bounded. 
\end{itemize}
In particular, every analytic $P$-ideal  is $F_{\sigma\delta}$.  Moreover, $\ideal$ is an $F_\sigma$ $P$-ideal, if,  and only if, there is a lscsm $\varphi$ such that $\ideal=\exh(\varphi)= \fin(\varphi)$.
\end{teo}

\subsection{Integer valued submeasures}
\label{integervalued}

$F_{\sigma}$ ideals are precisely the ideals of the form $\fin(\varphi)$ for some lscsm $\varphi$.  We recall this result  to point out  that such  $\varphi$ can have an extra property, which  we use later in our  discussion  of pathology of submeasures. 

We say that a lscsm $\varphi$ is {\em integer-valued} if it takes values in  $\N\cup\{\infty\}$ and $\varphi(\N)=\infty$.

\begin{teo}
\label{mazur}
(Mazur \cite{Mazur91}) For each $F_{\sigma}$ ideal $\ideal$ on $\N$, there is an integer-valued lscsm $\varphi$ such that $\ideal=\fin(\varphi)$.  Moreover, there is such $\varphi$ satisfying that $\varphi(\{n\})=1$ for all $n\in \N$.  
\end{teo}

\proof
We include a sketch in order to verify the last claim. Let $(\mathcal{K}_n)_n$ be a collection of closed hereditary subsets of $\cantor$ such that $\mathcal{K}_n\subseteq \mathcal{K}_{n+1}$, $A\cup B\in \mathcal{K}_{n+1}$ for all $A,B\in \mathcal{K}_n$ and $\ideal=\bigcup_n \mathcal{K}_n$.  We  can assume that $\mathcal{K}_0=\{\emptyset\}$, and since $\{\{n\}:\;n\in \N\}\cup\{\emptyset\}$ is a closed subset of $\cantor$, we can assume that $\{n\}\in \mathcal{K}_1$ for all $n\in \N$. Then the submeasure associated to $(\mathcal{K}_n)_n$ is  
given by $\varphi(A)=\min\{n\in \N:\; A\in \mathcal{K}_n\}$, if $\emptyset\neq A\in \ideal$, and  $\varphi(A)=\infty$, otherwise. 

\qed

We now highlight two relevant properties of integer-valued submeasures.
\
\begin{prop}
Let $\varphi$ be an unbounded integer-valued lscsm  such that  $\varphi(\{x\})=1$ for all $x\in \N$. Then 
\begin{enumerate}
\item[(i)] Every $n\in \N$ belongs to the range of $\varphi$. 
\item[(ii)]  For every integer $k\geq 2$, there is  a finite set $B$  such that $\varphi(B)=k$ and $\varphi(C)<k$ for all $C\subset B$ with $C\neq B$.  
\end{enumerate}
\end{prop}
\proof
(i) Suppose not and let  $n=\min(\N\setminus range(\varphi))$. Since $\varphi(\N)=\infty$, by the lower semicontinuity there is a  finite set $A$  such that  $\varphi(A)=\min\{\varphi(B): \;n<\varphi(B)\}$.  Let $m=\varphi(A)$. We can also assume that $A$ is $\su$-minimal with that property, that is, $\varphi(B)<n$ for all $B\su A$ with $B\neq A$.  Let $x\in A$, then 
\[
\varphi(A)\leq \varphi(A\setminus\{x\})+\varphi(\{x\})<n+1\leq m, 
\]
a contradiction. 

(ii) By (i) there is  a $\subseteq$-minimal finite set  $B$  such that $\varphi(B)=k$. 

\qed

\section{Pathology of submeasures}

The ideal $\suma(\varphi)$ is also induced by a measure determined by its values on  singletons. Namely, $\mu_{\varphi}(A)=\sum_{n\in A}\varphi(\{n\})$ is a measure and $\suma(\varphi)=\fin(\mu_{\varphi})=\exh(\mu_{\varphi})$. 
These ideals are the most well-behaved among  $F_{\sigma}$ or  $P$-ideals. Several good properties of summable ideals are shared by the larger  class of nonpathological ideals. 
We say that  $\mu $ is {\em dominated by $\varphi$}, if $\mu(A)\leq \varphi(A)$ for all $A$, in this case, we write $\mu\leq \varphi$. 
A lscsm $\varphi$ is {\em nonpathological} if it is the supremmum of all  ($\sigma$-additive) measures dominated by $\varphi$. 
A quite relevant application of this kind of submeasures was Farah's proof  that $\exh(\varphi)$ has the Radon-Nikod\'ym property if $\varphi$ is nonpathological \cite{Farah2000}. 
So far, nonpathology has been considered a property only about submeasures, we extend it to ideals as follows.

\begin{defi}
We say that an $F_{\sigma}$ ideal $\ideal$ is \emph{nonpathological} if there is a nonpathological lscsm $\varphi$ such that $\ideal=\fin(\varphi)$.
\end{defi} 

We show examples of ideals induced by both, a pathological and a nonpathological submeasure. The following proposition gives us a criterion for showing that an integer-valued submeasure is pathological.

\begin{prop}
\label{pathological}
Let $\varphi$ be an integer valued lscsm on a set $X$. Suppose there is a finite set $A\subseteq X$ with $|A|\geq 2$ such that $\varphi(A\setminus\{x\})<\varphi(A)$ for all $ x\in A$ and $\varphi(A)<|A|$.  Then $\varphi$ is pathological. 
\end{prop}

\proof
Suppose $\varphi$ is nonpathological. 
Let $m=|A|$ and  $x_0\in A$ be such that $\varphi(A\setminus\{x_0\})=\max\{\varphi(A\setminus\{x\}):\; x\in A\}$. 
Since $\varphi$ takes integer values and $\varphi(A)/m<1$,  $\varphi(A\setminus\{x_0\})+\varphi(A)/m < \varphi(A)$.
Pick a measure $\mu\leq \varphi$ such that 
$$
\varphi(A\setminus\{x_0\})+\varphi(A)/m\;<\mu(A)\leq \varphi(A).
$$
There is $y\in A$ such that $\mu(\{y\})\leq \varphi(A)/m$. Then 
\[
\mu(A\setminus\{y\})=\mu(A)-\mu(\{y\})\geq \mu(A)-\varphi(A)/m>\varphi(A\setminus\{x_0\})
\geq \varphi(A\setminus\{y\}),
\]
which contradicts that $\mu\leq \varphi$. 
\endproof

\medskip

Now we present a very elementary example of a pathological lscsm.
\begin{ex}
\label{patologica-minima}
{\em 
Let $\varphi$ be the lscsm  defined on $\{0,1,2\}$  by $\varphi(\emptyset)=0$, $\varphi(a)=1$ if $0<\vert a \vert\leq 2$ and $\varphi(\{0,1,2\})=2$. Then $\varphi$ is the minimal example of a pathological integer-valued submeasure on a finite set, where singletons have submeasure 1.  
}
\end{ex}

By elementary reasons, all $F_{\sigma}$ ideals and all analytic P-ideals can be induced by a pathological lscsm, as the following proposition shows.

\begin{prop}
\label{todostienepatologica}
Let $\varphi$ be any lscsm on $\N$. There is a pathological lscsm  $\psi$ such that  
$\fin(\varphi)=\fin(\psi)$ and $\exh(\varphi)=\exh(\psi)$.
\end{prop}

\proof 
Let $\varphi_0$ be the submeasure on $\{0,1,2\}$ defined in Example \ref{patologica-minima}. Let 
$$
\psi(A)=\varphi_0(A\cap\{0,1,2\})+\varphi(A\setminus\{0,1,2\}).
$$
$\psi$ is pathological, as $\psi\restriction \{0,1,2\}=\varphi_0$. Clearly $\psi$ works.
\endproof

Now we present an example of a pathological lscsm  $\varphi$ and a nonpathological lscsm $\psi$ such that 
$\fin(\varphi)=\fin(\psi)$ and $\fin(\varphi)$ is tall. 
 
\begin{ex}
\label{ED}
{\em Let   $(B_n)_n$ be a partition of $\N$ into infinite sets. 
The ideal $\mathcal{ED}$ is defined as the ideal generated by pieces and selectors of the partition $(B_n)_n$. Let $\mathcal{K}_0$ be the set $\{\emptyset\}$ and 
\[
\mathcal{K}_1= \{H\su\N: H\su B_n\;\text{for some $n$}\}\cup\{H\su \N:  \text{$H$ is a partial selector for $(B_n)_n$} \}.
\]
Then $\mathcal{K}_1$ is closed hereditary and $\mathcal{ED}$ is the ideal generated by $\mathcal{K}_1$. Let 
\[
\mathcal{K}_{n+1}= \{H_1\cup\cdots\cup H_{n+1}: H_i\in \mathcal{K}_n\; \text{for $1\leq i\leq n+1$}\}.
\]

Let $\varphi$ be the Mazur's submeasure defined in  the proof of Theorem \ref{mazur} for this family of closed hereditary sets.  Clearly $\mathcal{ED}=\fin(\varphi)$. 
We use Proposition \ref{pathological} to show that $\varphi$ is pathological. Pick a set $A=\{x_1, x_2, x_3\}$  such that $x_1\in B_0$ and $x_2, x_3\in B_1$. 
Notice that $\varphi(A)=2$ and  $\varphi(A\setminus\{y\})=1$ for all $y\in A$.

Consider the submeasure on $\N$ given by 
$$
\psi(A)=\min\{m\in\N :(\forall n> m)\vert A\cap B_n\vert\leq m\}.
$$
It is easy to see that $\fin(\psi)=\mathcal{ED}$. Now,  let us consider the family 
$$
\mathcal{S}=\{\mu_n^F: n\in\N, F\in[B_n]^{n+1}\}
$$
where, for each $n$ and $F$, $\mu_n^F$ is the counting measure supported on $F$.  Note that $\psi(A)=\sup\{ \mu(A):\mu\in\mathcal{S}\}$, since for every $A$ and every $m$ the following conditions are equivalent:
\begin{itemize}
    \item For all $n\geq m$, $\vert A\cap B_n\vert \leq m$, and
    \item For all $n$, and all $F\in [B_n]^{n+1}$, $\mu_n^F(A)\leq m$.
\end{itemize}
Hence, $\psi$ is nonpathological.

}
\end{ex}

For the sake of completeness,  we mention that both submeasures $\varphi$ and $\psi$ from the previous example remain pathological and nonpathological respectively, when they are restricted to
$\Delta=\bigcup_n C_n$, where each $C_n$ is a fixed subset of $B_n$ with cardinality $n+1$.  $\mathcal{ED}_{fin}$ denotes the restriction of $\mathcal{ED}$ to $\Delta$.  It is immediate to see that, in general,  every restriction of a nonpathological submeasure is nonpathological, while some restrictions of pathological submeasures are nonpathological.

\subsection{Degrees of pathology}
Farah's approach to pathology of submeasures on $\N$ includes a concept of degree of pathology  \cite{Farah2000}.  Associated to each lscsm $\varphi$, there is another lscsm $\widehat{\varphi}$, defined as follows. 
$$
\widehat{\varphi}(A)=\sup\{\mu(A):\mu \text{ is a measure dominated by }\varphi\},
$$
for all $A\subseteq \N$. Clearly $\widehat{\varphi}$ is the maximal nonpathological submeasure dominated by $\varphi$.

The \emph{degree of pathology}, which measures how far is a submeasure from being nonpathological, is defined by
$$
P(\varphi)=\sup\left\{\frac{\varphi(A)}{\widehat{\varphi}(A)}: \widehat{\varphi}(A)\neq 0 \; \& \; A\in \fin \right\}.
$$
Note that  $P(\varphi)=1$ if and only if $\varphi$ is nonpathological. Moreover, if $P(\varphi)=N<\infty$ then $\fin(\varphi)$ is equal to $\fin(\psi)$ for some  nonpathological submeasure $\psi$: In fact,  we have that  $\widehat{\varphi}\leq\varphi\leq N\widehat{\varphi}$, thus $\fin(\widehat{\varphi})\supseteq\fin(\varphi)\supseteq \fin(N\widehat{\varphi})$, and clearly  $\fin(\widehat{\varphi})= \fin(N\widehat{\varphi})$. 
Let us note that for every $F_{\sigma}$ ideal $\ideal$ and every $n>1$, there is a lscsm $\varphi$ such that  $\ideal=\fin(\varphi)$ and  $P(\varphi)\geq n$.  In fact, if  $\ideal=\fin(\varphi')$ is nonpathological, we can modify $\varphi'$, as we did in the proof of Proposition \ref{todostienepatologica}, by inserting a copy of the submeasure $\psi_{2n}$ defined in subsection \ref{examplemazur}. 
This proves that some nonpathological ideals may be defined by submeasures having arbitrarily large finite degrees of pathology.
Moreover, if $\varphi'$ is an arbitrary (in particular for nonpathological) submeasure and $A$ is an infinite set such that $\varphi'(A)<\infty$, we can modify $\varphi'$ by taking a submeasure $\psi$ on $A$ with $P(\psi)=\infty$ and $\psi(A)=\varphi'(A)$, and define $\varphi(B)=\max\{\psi(B\cap A),\varphi'(B\setminus A)\}$. 
Hence $\fin(\varphi)=\fin(\varphi')$ but $P(\varphi)=\infty$.
This construction proves that every $F_{\sigma}$ ideal properly containing $\fin$ is induced by a submeasure with infinite degree of pathology. 
However, this construction encapsulates the pathological part of the submeasure in a small set, what lefts the following question open.

\begin{question}
\label{absolut-patologico}
Is there a nonpathological $F_{\sigma}$ ideal $\ideal$ for which there is a lscsm $\varphi$ such that $\ideal=\fin(\varphi)$,  $P(\varphi)=\infty$ and $P(\varphi\upharpoonright A)<\infty$, for all $A\in\ideal$? 
\end{question}

In  light of the notion of degree of pathology, we can see that an $F_{\sigma}$ ideal $\ideal$ is pathological if and only if $P(\varphi)=\infty$ whenever $\varphi$ is a lscsm such that  $\fin(\varphi)=\ideal$. In the next section we present  an example of a pathological $F_{\sigma}$ ideal, as a particular case of a general method of constructing  pathological $F_{\sigma}$ ideals.

\subsection{Examples of pathological ideals}
K. Mazur constructed an $F_{\sigma}$ ideal which is not contained in any summable ideal \cite{Mazur91}. We show in this section that a  variation of Mazur's construction produces $F_{\sigma}$ ideals which are  not contained in any nonpathological ideal.

We need the concept of  a covering number similar to the one defined by J. Kelley  \cite{Kelley1959}. Given a finite set $K$, a covering $\mathcal{S}$ of $K$ and an element $i\in K$,  $B(i)$ denotes the cardinality of  $\{s\in\mathcal{S}:i\in s\}$ and $m(K,\mathcal{S})$ denotes the minimum of all  $B(i)$ for $i\in K$. The \emph{covering number of }$\mathcal{S}$ in $K$ is defined by
$$
\delta(K,\mathcal{S})=\frac{m(K,\mathcal{S})}{\vert \mathcal{S \vert}}.
$$

\begin{lema}
\label{lemakelleymazur}
Let $K$ be a finite set and $\mathcal{S}$ a covering of $K$. If $\pi$ is a probability measure on $K$, then there is $s\in\mathcal{S}$ such that $\pi(s)\geq \delta(K,\mathcal{S})$.
\end{lema}

\begin{proof}
Let us note that
$$\sum_{s\in\mathcal{S}} \pi(s)=\sum_{i\in K} B(i)\pi(\{i\})\geq m(K,\mathcal{S})\sum_{i\in K}\pi(\{i\})=m(K,\mathcal{S}).$$
Therefore, there is $s\in\mathcal{S}$ such that $\pi(s)\geq \frac{m(K,\mathcal{S})}{\vert\mathcal{S}\vert}=\delta(K,\mathcal{S})$.
\end{proof}

An \emph{interval partition} of $\N$ is a family $\{I_n:n\in\N\}$ of intervals of $\N$ such that $\min I_0=0$ and $\min I_{n+1}=\max I_n +1$. The next theorem is about submeasures defined  using an interval partition, but it can be stated in terms of a family of  pairwise disjoint finite sets which covers a given countable set.

\begin{teo}\label{teogeneralinterval}
Let $\varphi$ be a lscsm on $\N$ such that  there is  $M>0$ and an interval partition $\{I_n:n\in\N\}$ satisfying
\begin{itemize}
    \item the family $\mathcal{B}=\{A\subseteq \N:\varphi(A)\leq M\}$ covers $\N$,
    \item $\sup_n\varphi(I_n)=\infty$, and
    \item $\varphi(B)=\sup\{\varphi(B\cap I_n):n\in\N\}$, for all $B\subseteq\N$.
\end{itemize}
Let $\mathcal{S}_n$ be a subfamily of $\mathcal{P}(I_n)\cap\mathcal{B}$ such that $\mathcal{S}_n$ covers $I_n$, $\epsilon_n=\delta(I_n,\mathcal{S}_n)$ and  $\delta=\inf\{\epsilon_n:n\in\N\}$. If $\delta>0$,  then
\begin{enumerate}
    \item[(i)]  $\widehat{\varphi}$ is bounded,
    \item[(ii)]  $P(\varphi)=\infty$, 
    \item[(iii)] $\fin(\varphi)$ is not contained in any nontrivial nonpathological $F_{\sigma}$ ideal, and
    \item[(iv)] $\fin(\varphi)$ is a pathological $F_{\sigma}$ ideal.
\end{enumerate}
\end{teo}
\begin{proof}
Note that (ii) follows immediately from (i), and (iv) follows from (iii). Let us prove (i).
Let  $\mu$ be a finitely supported  measure dominated by $\varphi$. Then there are:
\begin{enumerate}
    \item a finite set $F\subseteq\N$,
    \item a probability measure $\pi_j$ on $I_j$, for each $j\in F$, and
    \item $\lambda_j>0$ for each $j\in F$ 
\end{enumerate}
such that the support of $\mu$ is contained in $\bigcup_{j\in F}I_j$ and $\mu\upharpoonright I_j=\lambda_j\pi_j$. 
By Lemma \ref{lemakelleymazur}, there is $s_j$ in $\mathcal{S}_j$ such that $\pi_j(s_j)\geq\epsilon_j$. Then, 
$$
0<\delta\sum_{j\in F} \lambda_j\leq \sum_{j\in F}\lambda_j\epsilon_j\leq\sum_{j\in F}\lambda_j\pi_j(s_j)=\mu(\bigcup_{j\in F}s_j)\leq \varphi(\bigcup_{j\in F} s_j)\leq M.
$$
Hence $\sum_{j\in F}\lambda_j\leq \frac{M}{\delta}$. Since $\mu(\bigcup_{j\in F}I_j)=\sum_{j\in F}\lambda_j\pi_j(I_j)=\sum_{j\in F}\lambda_j$, it follows that $\mu$ is bounded by $\frac{M}{\delta}$. Since $M$ and $\delta$ do not depend on $\mu$, we are done.

For (iii), 
we will prove that if $\psi=\sup \mathcal{M}$ for some family $\mathcal{M}$ of measures, and it is nontrivial in the sense that $\fin(\psi)\neq \mathcal{P}(\N)$,
then there is a set $B\in \fin(\varphi)$ which is not in $\fin(\psi)$.
For each $n>0$,  define $k_n\in\N$ recursively: $k_0=0$ and 
$$
k_{n+1}=\min\left\{j>k_n: \psi\left(\bigcup_{k_n\leq i<j}I_j\right)>n\right\}.
$$
Notice that $k_{n+1}$ is well defined since $\bigcup_{j\geq m}I_j\not\in\fin(\psi)$ for all $m$.
Let us pick a measure $\mu_n\in \mathcal{R}$ such that 
$$
\mu_n\left(\bigcup_{j=k_n}^{k_{n+1}-1}I_j\right)\geq n.
$$
By Lemma \ref{lemakelleymazur}, for each $k_n\leq j<k_{n+1}$, there is $s_j\in \mathcal{S}_j$ such that $(\mu_n\upharpoonright I_j)(s_j)\geq \delta \mu_n(I_j)$. 
Let us define $B_n=\bigcup_{j=k_n}^{k_{n+1}-1}s_j$.
Hence, 
$$\psi(B_n)\geq\mu_n\left(B_n\right)\geq\sum_{j=k_n}^{k_{n+1}-1}\delta\mu_n(I_j)=\delta\mu_n\left(\bigcup_{j=k_n}^{k_{n+1}-1}I_j\right)\geq\delta n.$$
On the other hand, $\varphi(B_n)\leq M$.
Hence $B=\bigcup B_n$ is in $\fin(\varphi)$ and is not in $\fin(\psi)$.
\end{proof}

\subsubsection{Mazur's example}
\label{examplemazur}

We present an example of a pathological ideal using Theorem \ref{teogeneralinterval}. Consider the following families of sets:

\begin{itemize}
    \item Let $K_n$ be  the set of all functions from $n$ to $m=2n$,  and
    \item $\mathcal{S}_n=\{\hat{i}:i=0,\dots, m-1\}$ where $\hat{i}=\{f\in K_n: i\notin range(f)\}$.
\end{itemize}

Note that for each $n>1$, $\mathcal{S}_n$ is a covering of $K_n$, and  no  subset of $\mathcal{S}_n$  with at most $n$ sets is a covering of $K_n$.
Moreover, every $f:n\to m$ avoids at least $n$ values in $m$. 
That proves that $m(K_n,\mathcal{S}_n)=n$, and since $\vert \mathcal{S}_n \vert=m$, we have that $\epsilon_n=\delta(K_n,\mathcal{S}_n)=\frac{1}{2}$. 
Let  $\psi_n$ be the subsmeasure on $K_n$ defined by
$$
\psi_n(A)=\min\{r:\exists b\in[m]^r A\subseteq\bigcup_{i\in b}\hat{i}\},
$$
for all $A\subseteq K_n$.
By identifying the sets $K_n$ with members of the corresponding interval partition, we can define the lscsm $\psi$ by
$$
\psi(A)=\sup\{\psi_n(A\cap K_n):n\in\N\}.
$$
Hence, we have that $\psi(K_n)=n+1$ and the family $\mathcal{B}=\{A\subseteq \N:\psi(A)\leq 1\}$ is a covering of $\N$. By Theorem \ref{teogeneralinterval}, $\mathcal{M}=\fin(\psi)$ is a pathological ideal. For the sake of completeness, let us note that $P(\psi_n)=\frac{n+1}{2}$.

\subsection{Submeasures with infinite pathological degree}
In this section we present a sufficient condition for a submeasure to have infinite pathological degree.

\begin{teo}\label{teogeneralincreasingpath}
Let $\varphi$ be an unbounded lscsm on $\mathbb{N}$. Suppose there is $M>0$ such that the family
$$\mathcal{B}=\{A\subseteq \N:\varphi(A)\leq M\}$$
is a covering of $\N$. Let $\{K_n:n\in\N\}$ be a strictly increasing sequence of finite sets such that $\bigcup K_n=\N$ 
and $\mathcal{S}_n$ a subfamily of $\mathcal{P}(K_n)\cap\mathcal{B}$ such that $\mathcal{S}_n$ covers $K_n$. Let 
$\delta_n=\delta(K_n,\mathcal{S}_n)$ and $\delta=\inf\{\delta_n:n\in\N\}$. If $\delta>0$, then $\widehat{\varphi}$ is bounded and  $P(\varphi)=\infty$. 
\end{teo}

\begin{proof}
Note that it is enough to prove that there is a uniform bound for all measures dominated by $\varphi$.
Let  $\mu$ be a finitely supported measure dominated by $\varphi$. Then, there are:
\begin{enumerate}
    \item $n\in\N$,
    \item a probability measure $\pi$ on $K_n$, and
    \item $\lambda>0$ 
\end{enumerate}
such that the support of $\mu$ is contained in $K_n$ and $\mu=\lambda\pi$. By Lemma \ref{lemakelleymazur}, there is $s$ in $\mathcal{S}_n$ such that  $\pi(s)\geq\delta_n$. Thus, 
$$
0<\lambda\delta\leq \lambda\delta_n\leq\lambda\pi(s)=\mu(s)\leq \varphi(s)\leq M.
$$
Hence $\lambda\leq \frac{M}{\delta}$. Since $\mu(K_n)=\lambda\pi(K_n)=\lambda$, it follows that $\mu$ is bounded by $\frac{M}{\delta}$. Since $M$ and $\delta$ do not depend from $n$, we are done.
\end{proof}

We remark that the collection $\mathcal{B}$ mentioned above is a covering of $\N$ iff $\{n\}\in \mathcal{B}$ for all $n\in\N$.
This requirement is easy to satisfy, for instance it holds if $\fin(\varphi)$ is tall.

\bigskip 

It remains open the following question:

\begin{question} Let $\varphi$ be a lscsm on $\N$ and $(K_n, \mathcal{S}_n)_n$ and $\mathcal{B}$ as in the hypothesis of Theorem \ref{teogeneralincreasingpath}. Suppose  $\widehat{\varphi}$ is bounded, is  $\delta>0$? If $\delta=0$, is  $\fin(\varphi)$  nonpathological? 
\end{question}

The calculation of $\delta$ for a given submeasure may not be easy to do. 
In the next section  we present two examples illustrating this computation. However,   we can naturally associate sequences $(K_n)_n$ and $(\mathcal{S}_n)_n$ to every lscsm given by Theorem \ref{mazur} and thus, in principle, we can calculate  the corresponding $\delta$.

\subsubsection{Solecki's ideal}
The ideal $\mathcal{S}$ is defined \cite{Solecki2000} on the countable set $\Omega$ of all clopen subsets of the Cantor set $2^{\N}$ whose  measure\footnote{The measure considered here is the product measure of $2^{\N}$ where $2$ is equipped with its uniform probability measure.} is equal to $\frac{1}{2}$ and it is generated  by the sets of the form $E_x=\{a\in \Omega: x\in a\}$ for $x\in\cantor$.
Hru\v{s}\'ak's Measure Dichotomy \cite{Hrusak2017} establishes that any pathological analytic P-ideal has a restrictions to a positive set which is  Kat\v{e}tov above  $\mathcal{S}$. Solecki's ideal is critical for the class of analytic P-ideals but is not a P-ideal. It is a tall $F_{\sigma}$ ideal but  we do not know if it is pathological, although, we show below that  there is a lscsm $\chi$ such that $\mathcal{S}=\fin(\chi)$ with $P(\chi)=\infty$.

Consider a lscsm $\chi$ on $\Omega$ given by 
$$
\chi(A)=\min\{n: \;\exists{x_1}, \cdots, x_n\in  2^{\N},\; \; A\subseteq E_{x_1}\cup\cdots \cup E_{x_n}\}.
$$
It is clear that $\mathcal{S}=\fin(\chi)$. We will show that it satisfies the hypothesis of Theorem \ref{teogeneralincreasingpath} and therefore $P(\chi)=\infty$.

Let us denote by $\langle s\rangle$ the clopen set $\{x\in 2^{\mathbb{N}}:s\subseteq x\}$, for $s\in 2^{<\omega}$.
For every $n>1$, we define $$
\Omega_n=\{b\in\Omega: (\forall s\in 2^n)(\langle s\rangle\subseteq b \text{ or } \langle s \rangle \cap b=\emptyset)\}.
$$ 
Notice that $\Omega_n$ is an increasing sequence of sets whose union is equal to $\Omega$: Let $a\in\Omega$, since $a$  is a compact set,   select a finite covering $\{\langle c_1\rangle,\dots \langle c_r\rangle\}$ of $a$ with $c_1,\dots c_r\in 2^{<\omega}$. Let $n$ be the length of the longest $c_j$. Then  $a\in \Omega_n$.

For a given $s\in 2^n$, let $\tilde{s}=\{b\in\Omega_n:\langle s\rangle \subseteq b\}$ and  $\mathcal{S}_n$ be  the family $\{\tilde{s}:s\in 2^n\}$.
Note that $\chi(\tilde{s})=1$ for all $s$. Moreover,  $\chi(A)=\min\{k:\exists s_1,\dots,s_k\in 2^n: A\subseteq\bigcup_{i=1}^k\tilde{s_i}\}$, for all $A\subseteq \Omega_n$. 

Note that for every $B\in[2^n]^{2^{n-1}}$, $\bigcup_{s\in B}\tilde{s}\neq \Omega_n$, while for every $C\in[2^n]^{2^{n-1}+1}$, $\bigcup_{s\in C}\tilde{s} = \Omega_n$. Thus $\chi(\Omega_n)=2^{n-1}+1$. 
Also note that each $a\in\Omega_n$ belongs exactly  to $2^{n-1}$ many sets in $\mathcal{S}_n$. Thus, $\delta(\Omega_n,\mathcal{S}_n)=\frac{1}{2}$ for all $n$.

\subsubsection{The ideal $\mathcal{ED}_{fin}$}

Recall that in   Example \ref{ED} we  show that $\mathcal{ED}_{fin}=\fin(\varphi)$  where $\varphi$ is defined on $\Delta=\bigcup_n C_n$, where $C_n$ is a subset of $B_n$ of size $n+1$, and  is given by $\varphi(A)=\sup\{\mu_n(A\cap C_n):n\in\N\}$ for  $\mu_n$ the counting measure on $C_n$. Consider
\begin{itemize}
    \item $K_n=\bigcup_{j\leq n}C_n$, 
    \item $\mathcal{S}_0=\{C_0\}$ and $\mathcal{S}_{n+1}=\{s\cup\{j\}:s\in\mathcal{S}_n \;\&\; j\in C_{n+1}\}$
\end{itemize}
for all $n$. Notice that $|\mathcal{S}_{n}|=(n+1)!$ 

In Example \ref{ED} we have shown that $\varphi$ is non pathological, i.e. $P(\varphi)=1$. On the other hand,   we show next that $\delta_n=\delta(K_n,\mathcal{S}_n)=\frac{1}{n+1}$, for all $n\geq 1$. 
Let $k\leq n$ and $i,j\in C_k$.
It is easy to verify that $B(i)=B(j)=(n+1)!/k$, i.e. $i$  belongs  to as many elements of $\mathcal{S}_n$ as $j$ does. Thus  $\delta(K_n,\mathcal{S}_n)=\frac{1}{n+1}$. 

\subsection{Degrees of pathology and Rudin-Keisler order}
For a given ideal $\idealj$ and a function $f:\N\to \N$, an ideal $f(\idealj)$  is defined  as follows: 
$$
f(\idealj)=\{A\su \N:f^{-1}(A)\in\idealj\}.
$$
An ideal $\ideal$ is said to be  \emph{Rudin-Keisler} below $\idealj$ if $\ideal=f(\idealj)$ for some $f$, denoted $\ideal\leq_{RK}\idealj$ and $f$ is called  a Rudin-Keisler reduction of $\idealj$ in $\ideal$. The \emph{Kat\v{e}tov order} is defined by $\ideal\leq_{K}\idealj$ if $\ideal\subseteq\ideal'$ for some $\ideal'\leq_{RK}\idealj$.
The analytic P-ideals have an elegant classification in the Kat\v{e}tov order, given by the Hrusak's Measure Dichotomy (Theorem 4.1 in \cite{Hrusak2017}). Such classification uses the degrees of pathology.

We now show how Rudin-Keisler order impacts on degrees of pathology.

\begin{lema}
\label{RKdegree}
Let $\ideal$ and $\idealj$ be $F_{\sigma}$ ideals, $\varphi$ a lscsm such that $\idealj=\fin(\varphi)$ and $f$ a Rudin-Keisler reduction of $\idealj$ in $\ideal$. Let $\varphi_f$ be defined by
$$
\varphi_f(A)=\varphi(f^{-1}(A))
$$
for all $A\subseteq\N$.
Then, the following hold.
\begin{enumerate}
    \item $\varphi_f$ is a lscsm and $\ideal=\fin(\varphi_f)$,
    \item if $\varphi$ is a measure then $\varphi_f$ is also a measure,
    \item if $\nu$ is a measure dominated by $\varphi$ then $\nu_f$ is a measure dominated by $\varphi_f$, 
    \item for all $A\subseteq \N$, $\widehat{\varphi_f}(A)\geq\widehat{\varphi}(f^{-1}(A))$, and
    \item $P(\varphi_f)\leq P(\varphi)$.
\end{enumerate}
\end{lema}
\begin{proof}
(1)-(3) are routine. For (4), $\widehat{\varphi_f}(A)=\sup\{\nu(A):\nu\text{ is a measure dominated by }\varphi_f\}\geq\sup\{\mu_f(A):\mu\text{ is a measure dominated by }\varphi\}=\sup\{\mu(f^{-1}(A)):\mu\text{ is a measure dominated by }\varphi\}=\widehat{\varphi}(f^{-1}(A))$, for all $A\subseteq\N$. 
For (5) we use (4) to argue that for all $A\subseteq\N$,
$$\frac{\varphi_f(A)}{\widehat{\varphi_f}(A)} \leq \frac{\varphi(f^{-1}(A))}{\widehat{\varphi}(f^{-1}(A))}\leq P(\varphi).$$
Hence, $P(\varphi_f)\leq P(\varphi)$.
\end{proof}

\begin{teo}\label{teorkpato}
Let $\ideal$ and $\idealj$ be $F_{\sigma}$ ideals such that $\ideal\leq_{RK}\idealj$. If $\idealj$ is nonpathological then $\ideal$ is nonpathological.
\end{teo}

\begin{proof}
It follows immediately from Lemma \ref{RKdegree} and  the fact that an ideal is non pathological if it has a submeasure with finite pathological degree. 
\end{proof}

Recently, Figueroa and Hrusak \cite{FigueroaHrusak2023} proved that nonpathological $F_{\sigma}$ ideals (and every restriction of them) are Katetov-below $\mathcal{Z}$, the ideal of asymptotic density zero sets. This implies that $\mathcal{S}$ is pathological. However we do not know if $\mathcal{S}$ plays a critical role among pathological $F_{\sigma}$ ideals in Katetov or Rudin-Keisler orders.

\begin{question}
How is $\mathcal{S}$ related (Kat\v{e}tov, Rudin-Keisler) with pathological $F_{\sigma}$ ideals?
\end{question}

As we show next, Mazur's ideal $\mathcal{M}$ defined in subsection \ref{examplemazur} confirms the critical role played by  Solecki's ideal $\mathcal{S}$ in the Kat\v{e}tov order for the collection of pathological ideals.

\begin{teo}
\label{ReductionToS}
There is an $\mathcal{M}$-positive set $X$ such that
$\mathcal{S}\leq_{RK} \mathcal{M}\upharpoonright X$.
\end{teo}
\begin{proof}
Recall that $K_n$ is the collection of all functions from $n$ to $2n$ and $\mathcal{S}_n=\{\hat{i}:\;i=0,\dots, m-1\}$. 
Let $X_n$ ($n\geq 0$) be the set of all one-to-one functions in $K_n$, and $X=\bigcup_{n}X_{2^n}$. 
Note that $X$ is $\mathcal{M}$-positive since $X\cap K_{2^n}$ cannot be covered by less or equal than $2^n$ members of $\mathcal{S}_{2^n}$, for all $n$.
For each $n$, fix an enumeration $\{s^n_0,s^n_1, \dots,s^n_{2^n-1}\}$ of $2^n$.
The Rudin-Keisler reduction $f$ from $X$ to $\Omega$ works as follows. If $r\in X_{2^n}$ then 
\[
f(r)=2^{\N}\setminus\bigcup_{j=0}^{2^n-1}\langle s^{n+1}_{r(j)}\rangle.
\]
Note that $f(r)\in\Omega_{n}$.

We show first, that if $A\in\mathcal{S}$ then $f^{-1}(A)\in\mathcal{M}\upharpoonright X$. Suppose  $A=\{a\in\Omega:x\in a\}$ for some $x\in 2^{\N}$.
Then for all $n$, $A\cap\Omega_{n}=\{a\in \Omega_n:\; \langle x\upharpoonright n\rangle \subseteq a\}$, but for some $j_n\in\{0,\dots,2^{n}-1\}$ it happens that $x\upharpoonright n=s^{n}_{j_n}$. Hence $f^{-1}(A)\cap X_n$ is the set $\widehat{j_n}=\{r\in X_n: j_n\notin range(r)\}$, what proves that $\psi(f^{-1}(A))=1$. 

On the other hand, suppose that $f^{-1}(A)\in\mathcal{M}$.
Then there exists $N$ such that for all $n$, there is $J_n\in[2^{n}]^{N}$ such that $r\cap J_n=\emptyset$ for all $r\in f^{-1}(A)\cap K_{2^n}$.
Thus for all $n$ there are 
$c^{n}_{j_1},c^{n}_{j_2},\dots,c^{n}_{j_N}\in 2^{n}$ such that $\langle c^{n}_{j_i}\rangle \subseteq A\cap \Omega_n$, for some $i=1,\dots, N$. Hence, $\chi(f^{-1}(A))\leq N$, what proves that $A\in\mathcal{S}$.
 \end{proof}

\section{Tallness and Borel selectors.}
\label{tallness}

In this section we address the question of the tallness of $\fin(\varphi)$. As a first remark, we notice that there is a simple characterization of when $\exh(\varphi)$ is tall. In fact,  $\exh(\varphi)$ is tall iff  $\suma(\varphi)$ is tall iff $\varphi(\{n\})\rightarrow 0$.  Indeed, if we let $C_{n}=\{x\in \N:\; 2^{-n}<\varphi(\{x\})\}$, then  $\exh(\varphi)$ is tall iff each $C_n$ is finite.   Notice also that   $\fin(\varphi)$ is tall, whenever $\exh(\varphi)$ is tall, as  $\exh(\varphi)\subseteq\fin(\varphi)$. But the converse is not true, that is, it  is possible that $\fin(\varphi)$ is tall while $\suma(\varphi)$ and $\exh(\varphi)$ are not (see Example \ref{ejemadecuada}).

On the other hand, Greb\'{\i}k and Hru\v{s}\'{a}k \cite{GrebikHrusak2020} showed that there are  no simple characterizations  of the class of tall  $F_{\sigma}$ ideals, in fact, they showed that the collection of  closed subsets of $\cantor$ which generates an $F_\sigma$ tall ideal is not Borel as a subset of the hyperspace $K(\cantor)$. 

\subsection{Property A and ideals generated by homogeneous sets of a coloring}
\label{propiedadA}

In this section we are going to examine two very different  conditions implying tallness of an $F_{\sigma}$ ideal. 
We  introduce a property weaker than requiring that  $\lim_n \varphi (\{n\})=0$ but which suffices to get that $\fin(\varphi)$ is tall.  

\begin{defi}
A lscsm $\varphi$ on $\N$ has  {\em property A}, if $\varphi(\N)=\infty$ and  $\varphi(\{n\in \N:\; \varphi(\{n\})>\varepsilon\})<\infty$ for all $\varepsilon>0$. 
\end{defi}

Property A can be seen as a condition about  the convergence of $(\varphi(\{n\}))_n$ to 0, but in a weak sense. In fact, let us recall that for a given  filter $\mathcal{F}$ on $\N$,  a sequence $(r_n)_n$  of real numbers {\em $\mathcal{F}$-converges} to $0$, if $\{n\in \N: |r_n|<\varepsilon\}\in \mathcal{F}$ for all $\varepsilon>0$.  Let $\mathcal{F}$ be the dual filter of $\fin(\varphi)$. Then $\varphi$ has property A iff $(\varphi(\{n\}))_n$ $\mathcal F$-converges to $0$.

On the other hand, property A has also a different interpretation. We recall that an ideal $\ideal$ over $\N$ is {\em weakly selective} \cite{HMTU2017},  if given a positive set $A\in \ideal^+$ and a partition $(A_n)_n$ of $A$ into sets in $\ideal$, there is $S\in\ideal^+$ such that $S\cap A_n$ has at most one point for each $n$. 
A submeasure $\varphi$ has property A, if $\fin(\varphi)$ fails to be weakly selective  in the following  partition of $\N$:  $A_{n+1}=\{x\in \N: 1/2^{n+1} \leq\varphi (\{x\}) < 1/2^n\}$ and $A_0= \{x\in \N: 1\leq \varphi(\{x\})\}$. In fact, any selector  for $\{A_n:\; n\in \N\}$ belongs to  $\exh(\varphi)$ and thus to $\fin(\varphi)$.

\begin{prop}
\label{propAtall}
$\fin(\varphi)$ is tall  for all lscsm $\varphi$ with property A. 
\end{prop}

\proof Let $A\su \N$ be an infinite set.   If there is $\varepsilon>0$ such that $A\su  \{n\in \N:\; \varphi(\{n\})>\varepsilon\}$, then $A\in \fin(\varphi)$ as $\varphi$ has property A. Otherwise, pick  $n_k\in A$ such that $\varphi(\{n_k\})\leq 2^{-k}$ for all $k\in \N$. Let $B=\{n_k:\; k\in \N\}$. Then $\varphi(B)\leq \sum_k \varphi (\{n_k\})<\infty$.  
\qed

Note that any integer valued lscsm fails to  have the property A.  Thus,  every $F_\sigma$ tall ideal $\ideal$ is induced by a lscsm $\psi$ without the property A (for example,  the one given by the proof of Mazur's theorem applied to $\ideal$). 

Now we present a natural construction of submeasures with property A.   In particular, it provides  a lscsm $\varphi$ such that $\fin(\varphi)$ is tall but $ \varphi (\{n\})\not\to 0$.

\begin{ex}
\label{ejemadecuada} 
{\em 
Let $\{B_n:n\in\N\}$ be a partition of $\N$ into infinite sets and  $ \{B_n^k:k\in\N\}$ be  a partition of $B_n$ satisfying:
\begin{itemize}
\item $B_n^0$  consists of the first $2^n(n+1)$ elements of $B_n$.

\item $\min B_n^{k+1}=\min\{x\in B_n:x>\max B_n^k\}$.

\item $\vert B_n^{k+1}\vert \geq \vert B_n^k\vert$.

\end{itemize}
\medskip

\noindent Let  $\nu_n^k$ be the measure on  $B_n^k$ given by $\nu_n^k (\{x\})=\frac{n+1}{\vert B_n^k\vert}$ for all $x\in B_n^k$. Let 
$$
\varphi_n=\sup_k \nu_n^k$$ 
and 
$$
\varphi=\sum_n \varphi_n.
$$ 
Then $\varphi$ is a nonpathological lscsm. We list some useful facts about this construction.

\medskip
	
\begin{enumerate}
\item $\varphi (B_n)=\varphi (B_n^k)=n+1$ for all $n$ and $k$. 		

\item Let $(n_i)_i$ and $(k_i)_i$ two sequences in $\N$. Suppose $(n_i)_i$ is increasing. Then $\varphi(\bigcup_i B_{n_i}^{k_i})=\infty$.		

\item  $\varphi$ has property A.  Let $\varepsilon>0$ and  $M_{\varepsilon}=\{x\in\N: \varphi(\{x\})\geq\varepsilon\}$.
Notice  that  $\nu_n^k (\{x\})\leq \frac{1}{2^n}$ for all $x\in B_n^k$.
Let $N$ be such that $2^{-N}<\varepsilon$, then $M_{\varepsilon}$ is disjoint from $B_{m}$ for all $m>N$ and  thus $M_{\varepsilon}\subseteq B_0\cup \cdots\cup B_N$ belongs to $\fin(\varphi)$.

\item $\fin (\varphi)$ is not a P-ideal.  In fact, the $P$-property fails at $(B_n)_n$. Indeed, let us suppose that  $B_n\subseteq^*X$ for all $n$. Then for each $n$ there is $k_n$ such that $B_n^{k_n}\subseteq X$, and thus $\bigcup_nB_n^{k_n}\subseteq X$. By (2), $X\notin\fin(\varphi)$.

\item Every selector of the $B_n$'s belongs to $\suma(\varphi)$.

\item $B_n\in \fin(\varphi)\setminus \exh(\varphi)$ for all $n$. 

\end{enumerate}

\bigskip

Now we present two particular examples of the previous general construction. 

\medskip

\begin{itemize} 
\item[(a)] Suppose  $\vert B_n^{k+1}\vert=\vert B_n^k\vert$ for all $n$ and $k$. Notice that $\nu_n^k (\{x\})=\frac{n+1}{\vert B_n^k\vert}=1/2^n$ for all $x\in B_n$. Thus $\varphi(\{x\})=1/2^n$ for all $x\in B_n$ and  $\lim_n \varphi(\{n\})$  does not exist. Then $\suma(\varphi)$ and $\exh(\varphi)$ are not tall,  but $\fin(\varphi)$ is tall since it has the property A. 

\medskip

\item[(b)] Suppose   $\vert B_n^{k+1}\vert=\vert B_n^k\vert+n+1=(n+1)(2^n+k)$.
Then $\nu_n^k (\{x\})=\frac{n+1}{\vert B_n^k\vert}=\frac{1}{2^n+k}$ for all $x\in B_n^k$. 
We show that  $\varphi(\{m\})\to 0$, when $m\to \infty$. 
Given $\varepsilon>0$, we have seen that  $M_{\varepsilon}=\{x\in\N: \varphi(\{x\})\geq\varepsilon\}$ is disjoint from $B_{m}$ for all $m>N$ when  $2^{-N}<\varepsilon$, and it is also disjoint from $B_n^k$ when $k^{-1}<\varepsilon$. Hence $M_{\varepsilon}$ is finite. 

We claim that $\suma(\varphi)\neq \exh(\varphi)\neq \fin(\varphi)$.  By (4) and (6), it is sufficient to prove that there is $X\in\exh(\varphi)\setminus \suma(\varphi)$. 
For a fixed $n$, let $X=\{x_k:k\in\N\}$ be such that $x_k\in B_n^k$.  Since $\varphi(\{x_k\})=\frac{1}{2^n+k}$ for all $k$,  $X\not\in \suma(\varphi)$. On the other hand, $\varphi(\{x_k:\, k\geq m\})=\frac{1}{2^n+m}\to 0$ when $m\to \infty$.

\end{itemize}

\medskip

In  both  examples (a) and (b), $\varphi$ is a nonpathological submeasure since it can be expressed as $\sup\{ \mu_s:s\in\mathbb{N}^{<\omega}\}$, where $\mu_s$ is defined by
$$\mu_s(D)= \sum_{j<\vert s \vert} \nu_{j}^{s(j)}(D\cap B_{j}^{s(j)})$$
for  $s\in\mathbb{N}^{<\omega}$ and $D\subseteq \mathbb{N}$. 
}
\end{ex}

Now we present some examples of tall ideals which do not contain $\fin(\varphi)$ for  any $\varphi$ with property A.  Our examples are motivated by Ramsey's theorem. We refer the reader to  section \ref{preliminares} where the notation is explained and to \cite{GrebikUzca2018} for more information about this type of ideals

 We say that a coloring $c:[X]^2\to 2$ \emph{favors  color} $i$, if there  are no infinite $(1-i)$-homogeneous sets and in every set belonging to  $Hom(c)^*$   there are $(1-i)$-homogeneous sets of any finite cardinality. 

\begin{prop}
\label{favor0}
Let $c:[X]^2\to 2$ be a coloring that favors a color. Then, $Hom(c)$ does not contain $\fin(\varphi)$  for any  lscsm $\varphi$ with property A.
\end{prop}

\begin{proof}
Suppose $c$ favors color 0. Let $\varphi$ be an arbitrary  lscsm on $X$ with property A and suppose that $\fin(\varphi)\su Hom(c)$.  We will construct a set $A$ in $Hom(c)^+$  with $\varphi(A)<\infty$, which is a contradiction.  Let $B_{n}=\{x\in \N:\; 2^{-n}<\varphi(\{x\})\}$ for each $n\in \N$. As $\varphi$ has property $A$,  $B_n\in \fin(\varphi)$ and  $\N\setminus B_n\in Hom(c)^*$. By hypothesis, for each $n\in \N$, there is a 1-homogeneous finite set $A_n$  with $n$ elements and  such that $A_n\cap  B_n=\emptyset$. 
Since $\varphi(A_n)\leq \frac{n}{2^ n}$,  $A=\bigcup_n A_n\in \fin(\varphi)$.  As $\fin(\varphi)\su Hom(c)$, there is a  finite union of 0-homogeneous sets $C=C_1\cup\cdots\cup C_k$  such that $A\subseteq^* C$. As $A$ is infinite,  there is $l>k$ such that  $A_{l} \su C$. Since $A_l$ has $l$ elements,  there are $i\leq k$ and $x\neq y\in A_{l} \cap C_i$ which is imposible as $A_{l}$ is 1-homogeneous and $C_i$ is 0-homogeneous. A contradiction. 
\end{proof}

We present two examples of colorings satisfying  the hypothesis of the previous proposition. 

\begin{ex}
\label{edfin}
Let $(P_n)_n$ be a  partition of  $\N$ such that $|P_n|=n$.  
Let $c$ be the coloring given by $c\{x,y\}=0$ iff $x,y\in P_n$ for some $n$.
This coloring favors color 1. 
Notice that   $\mathcal{ED}_{fin}$ is the ideal generated by the $c$-homogeneous sets. 
\end{ex}

\begin{ex}
\label{SI}
Let $\Q$ be the rational numbers in $[0,1]$. Let $\{r_n:\; n\in\N\}$ be an enumeration of $\Q$. The {Sierpinski coloring}, $c:[\Q]^2\to 2$, is defined by $c\{r_n, r_m\}=0$ if $n<m$ iff  $r_n<r_m$. Denote by $\mathcal{SI}$ the ideal generated by the $c$-homogeneous sets.  Observe that the homogeneous sets are exactly the monotone subsequences of $\{r_n:\; n\in \N\}$. 
For each $n$, pick $X_n\su (n,n+1)$ an infinite homogeneous set of color 0 and let $X=\bigcup_n X_n$.  Then $X\in Hom(c)^+$.  It is easy to check that  $c\restriction X$ favors color $0$.   
\end{ex}

To see that Proposition \ref{favor0} is not  an equivalence, we notice that $\mathcal{ED}$ is the ideal generated by the homogeneous sets of a coloring not favoring any color, nevertheless,  by an argument similar to that used in the proof of Proposition \ref{favor0}, it can be shown that   $\mathcal{ED}$  does not contain $\fin(\varphi)$  for any  lscsm $\varphi$ with property A.

\subsection{Katetov reduction to $\mathcal{R}$}
Recall that $\mathcal{R}$ denotes the random ideal which is  generated by the collection of cliques and independent sets of  the random graph.
It is easy to check that if $\mathcal{R}\leq_K\ideal$, then $\ideal$ is tall. As we said before, for a while it was conjecture that any Borel tall ideal was Katetov above $\mathcal{R}$. This was shown to be false in \cite{GrebikHrusak2020}. In fact, they proved that there are $F_\sigma$ tall ideals which are not Katetov above the random ideal.  In this section we present a partial answer to the question of whether $\mathcal{R}\leq_K \ideal$ for a nonpathological tall ideal $\ideal$.

By the universal property of the random graph the following well known fact is straightforward.

\begin{prop}
\label{Rmin}
Let $\ideal$ be an ideal on $\N$. Then, $\mathcal{R}\leq_K \ideal$ if and only if there is $c:[\N]^2\to 2$ such that $\hom(c)\subseteq \ideal$. 
\end{prop}

The next fact illustrates the previous observation. 

\begin{prop}
$\mathcal{R}\leq_K \fin(\varphi)$ for all lscsm   $\varphi$ with property A.  
\end{prop} 

\proof 
Let
\[
B_{k+1}=\{n\in \N:\; 2^{-k-1}\;< \varphi(\{n\})\leq 2^{-k}\}
\]
for $k\in \N$, $B_0=\{n\in \N:\; \varphi(\{n\})>1\}$ and $B_{-1}=\{n\in \N: \varphi(\{n\})=0\}$. Let $c:[\N]^{2}\to \{0,1\}$ be the coloring  associated to the partition $\{B_k: \;k\in \N\cup\{-1\}\}$, that is to say, $c(\{n,m\})=1$ if, and only if,  there is $k$ such that $n,m\in B_k$. Then $A\in \hom(c)$ iff $A\subseteq B_{k}$ for some $k\in \N\cup\{-1\}$  or $|A\cap B_k|\leq 1$ for all $k$. Notice that $B_{-1}\in \fin(\varphi)$ and the argument used in the proof of Proposition \ref{propAtall}  shows that  $\hom(c)\subseteq \fin(\varphi)$.
\endproof

The next definition was motivated by the results presented in the last section about a representation of $F_{\sigma}$ ideals in  Banach spaces. 

\begin{defi}
\label{type-c0} A sequence of measures  $(\mu_k)_k$  on $\N$ is of type $c_0$ if  $\lim_k\mu_k(\{n\})=0$ for all $n$.
 \end{defi}

The main result of this section is the following:

\begin{teo}
\label{c0tall}
Let $\varphi=\sup_k\mu_k$ be a  nonpathological lscsm where  $(\mu_k)_k$ is a sequence of measures on $\N$ of type $c_0$.  The following are equivalent. 

\begin{itemize}
\item[(i)] $\mathcal{R}\leq_K \fin(\varphi)$.

\item[(ii)] $\fin(\varphi)$ is tall. 
\item[(iii)] $\lim_n\mu_k(\{n\})=0$ for all $k$ and $\sup_n\varphi(\{n\})<\infty$. 
\end{itemize}
\end{teo}

For the proof we need some auxiliary results. 
Let $(\mu_k)_k$ be a sequence of measures on $\N$ such that   $\mu_k(\{n\})\leq 1$ for all $n$ and $k$.  For $n,i\in \N$, let
\begin{equation}
\label{partition}
A^n_i=\left\{k\in \N: \; 2^{-i-1}< \mu_k(\{n\}) \leq 2^{-i}\right\}\;\;\mbox{and}\;\;   A^n_\infty=\{k\in\N:\; \mu_k(\{n\})=0\}.
\end{equation}
Each  $\mathcal{P}_n=\{A^n_i:\; i\in\N\cup\{\infty\}\}$ is a partition of $\N$ (allowing that some pieces of a partition are empty).  Let $L_n=\{i\in \N:\; A^n_i\neq \emptyset\}$. We  now show  that, without loss of generality, we can assume  that each $L_n$ is finite. 

\begin{lema}
\label{finitesupport}

Let $\varphi=\sup_k\mu_k$ be a nonpathological lscsm where $(\mu_k)_k$ is a sequence of measures on $\N$ such that   $\mu_k(\{n\})\leq 1$ for all $n$ and $k$.    There is another sequence of measures $(\lambda_k)_k$ on $\N$ such that 
\begin{itemize}
\item[(i)] $\lambda_k\leq \mu_k$ for all $k$.

\item[(ii)]  $\fin(\varphi)=\fin(\psi)$ where $\psi=\sup_k \lambda_k$.

\item[(iii)] If $(\mathcal{P}_n)_n$ is the sequence of partitions associated to $(\lambda_k)_k$, then each $L_n$ is finite. 
\end{itemize}
\end{lema}

\proof
Let  $\lambda_k(\{n\})=i/2^{n}$, if $i/2^{n} < \mu_k(\{n\})\leq (i+1)/2^{n}$ for some $0< i< 2^n$, and $\lambda_k(\{n\})=0$ otherwise. 
Let $\psi=\sup_k\lambda_k$.  Clearly  $\lambda_k\leq \mu_k$. Thus $\fin(\varphi)\subseteq \fin(\psi)$. 
Notice that  $|\mu_k(\{n\})-\lambda_k(\{n\})|\leq 1/2^n$ for all $n$ and $k$. Let $F\subseteq \N$ be a finite set. Then 
\[
\mu_k(F)= (\mu_k(F)-\lambda_k(F))+\lambda_k(F)\leq \sum_{n\in F}1/2^n +\psi(F).
\]
Thus $\fin(\psi)\subseteq \fin(\varphi)$. Let $\{A^n_i:\; i\in\N\cup\{\infty\}\}$ be the sequence of partitions  associate to $(\lambda_k)_k$.
It is immediate that  $A^n_i\neq \emptyset$,  only if  $i< n$.
\endproof

\begin{lema}
\label{tall-weaklynull}

Let $\varphi=\sup_k\mu_k$ be a nonpathological lscsm where each $\mu_k$ is a measure on $\N$. If $\fin(\varphi)$ is tall, then $\lim_n\mu_k(\{n\})=0$ for all $k$  and $\sup_n\varphi(\{n\})<\infty$.
\end{lema}
\proof For the first claim, fix $k\in \N$ and consider $B_i=\{n\in\N:\; \mu_k(\{n\})>1/i\}$ for $i\geq 1$.  Suppose the conclusion is false.  Then $B_i$ is infinite for some $i$. Since $\fin(\varphi)$ is tall, there is $B\subseteq B_i$ infinite such that $\varphi(B)<\infty$. 
However, for every $F\subseteq B$ finite we have that $|F|/i< \mu_k(F)\leq \varphi(F)\leq \varphi(B)$, which contradicts that $\varphi(B)<\infty$. 

For the second claim, suppose it is false and  let $(n_k)_k$ be an increasing sequence in $\N$ such that $k\leq \varphi(\{n_k\})$. Clearly every infinite subset of $\{n_k:\; k\in \N\}$ does not belong to $\fin(\varphi)$, which contradicts the tallness of $\fin(\varphi)$.
\endproof

To each sequence of partitions $\mathcal{P}_n=\{A^n_i:\; i\in\N\cup\{\infty\}\}$ of $\N$ we associate a coloring  $c:[\N]^2\to 2$ as follows:  for $n<m$
\begin{equation}
\label{coloring2a}
c\{n,m\} =1 \;\text{iff} \; ( \forall i\in\N) \;(A^n_i\subseteq  A^m_\infty \cup \bigcup_{j=i+1}^{\infty} A^m_j).
\end{equation}

\begin{lema}
\label{color1}
Let $\varphi=\sup_k\mu_k$ be a nonpathological lscsm where $(\mu_k)_k$ is a sequence of measures on $\N$ such that  $\mu_k(\{n\})\leq 1$ for all $n$ and $k$. Let $\{A^n_i: \; i\in \N\cup\{\infty\}\}$, for $n\in \N$, be the sequence of partitions associated to $(\mu_k)_k$ and $c$ be  the associate coloring. Every $c$-homogeneous infinite set of color 1 belongs to $\fin(\varphi)$.
\end{lema}

\proof
Let $H=\{n_i:\;i\in\N\}$ be the increasing enumeration of an infinite homogeneous set of color 1.  We claim that for all $k$ and $m$ 
\[
\mu_k(\{n_0, \cdots, n_m\})=\sum_{i=0}^m \mu_k(\{n_i\})\leq 2,
\]
which implies that  $\varphi(H)\leq 2$. In fact, fix $k$ and $m$. We can assume without loss of generality that 
$\mu_k(\{n_i\})\neq 0$ for all $i\leq m$.  For each $i\leq m$, let $j_i$ be such that $k\in A^{n_i}_{j_i}$. Since $H$ is 1-homogeneous and  $k\in A^{n_0}_{j_0}$ and  $k\in A^{n_1}_{j_1}$, we have  $j_0<j_1$. In general, 
we conclude that $j_0<j_1<\cdots <j_m$.
Thus
\[
\displaystyle\sum_{i=0}^m \mu_k(\{n_{i}\})\leq \sum_{i=0}^\infty\frac{1}{2^{i}}.
\]
\endproof

\medskip

\noindent
{\em Proof of Theorem \ref{c0tall}:} Clearly (i) implies (ii).   Lemma \ref{tall-weaklynull} shows that (ii) implies (iii). 

Suppose (iii) holds. Let $M=\sup_n\varphi(\{n\})$. Using $\mu_k/M$ instead of $\mu_k$ we can assume that $\mu_k(\{n\})\leq 1$ for all $n$ and $k$.  By Lemma \ref{finitesupport}, we can also assume that the partition $\{A^n_i:\; i\in\N\cup\{\infty\}\}$  associated to $(\mu_k)_k$ satisfies that $L_n=\{i\in \N:\; A^n_i\neq \emptyset\}$ is finite for all $n$.
Let $c$ be the coloring given by \eqref{coloring2a}. We will show that $\hom(c)\su\fin(\varphi)$, then (i) follows  by Proposition \ref{Rmin}. By Lemma \ref{color1}, it suffices to show that every infinite $c$-homogeneous set is of color 1. That is, we have to show that for all $n$, there is $m>n$ such that $c\{n,l\}=1$ for all $l\geq m$.

Fix $n\in \N$. Since $(\mu_k)_k$ is of type $c_0$,  $A^n_i$ is finite for all $i\in \N$. Consider the following set:
$$
S=\{k\in\N: \; \mu_k(\{n\})\neq 0\}=\bigcup\{A^n_i: i\in L_n\}.
$$
As $L_n$ is finite,  $S$ is also finite. So,  
let $i_0$ be such that $2^{-i_0}<\mu_k(\{n\})$ for all $k\in S$. 
Since $\lim_m\mu_k(\{m\})=0$ for all $k$, there is $m>n$ such that
\[
(\forall l\geq m)(\forall k\in S)(\; \mu_k(\{l\})<2^{-i_0}\;).
\]
 That is to say, for all  $k\in S$ and all $l\geq m$, there $j>i_0$ so that $k\in A^l_j$. Thus $c\{n,l\}=1$.  
\endproof

\subsection{Borel selectors}
\label{BS}
Recall that a tall family  $\ideal$ admits a {\em Borel selector}, if there is a Borel function $F: [\N]^\omega\to  [\N]^\omega$  such that $F(A)\su A$ and  $F(A)\in \ideal$ for all $A$.   This notion was studied in \cite{GrebikHrusak2020,GrebikUzca2018,grebik2020tall}.  
In this section we analyze Borel selectors for nonpathological tall $F_{\sigma}$ ideals. 

The typical examples of families with Borel selectors are the collection of homogeneous sets.   For instance, $\mathcal{R}$ admits a Borel selector and thus  if $\mathcal{R}\leq_K \ideal$, then $\ideal$ has a Borel selector. More generally, notice that if $\ideal\leq_K \idealj$ and $\ideal$ has a Borel selector, then so does $\idealj$.   The collection of all $K\in K(\cantor)$ such that the ideal generated by $K$ has a Borel selector is a $\mathbf{\Sigma}^1_2$ (see \cite{GrebikUzca2018}), so it has the same complexity as the collection of codes for $F_{\sigma}$ ideals which are  Kat\v{e}tov above $\mathcal{R}$ \cite{GrebikHrusak2020}.  We do not know an example of an $F_{\sigma}$ ideal admitting a Borel selector which is not Kat\v{e}tov above $\mathcal{R}$. 
For concrete examples of $F_\sigma$ tall ideals without Borel selectors see \cite{grebik2020tall}.

Let   $\mathcal{Q}_n$, for each $n\in \N$,   be a  collection of pairwise disjoint subsets $\N$, say $\mathcal{Q}_n=\{B^n_i:\; i\in \N\}$. The sequence $(\mathcal{Q}_n)_n$ is   {\em eventually disjoint},  if there is $p\in \N$ such that  for all $n\neq m$
\begin{equation}
\label{evenDis}
\,B^n_i \cap B^m_i=\emptyset,\;\;\mbox{for all $i>p$}.
\end{equation}

\begin{teo}
\label{evenDis3}
Let $\varphi=\sup_k\mu_k$ where  each $\mu_k$ is a measure on $\N$ such that  $\mu_k(\{n\})\leq 1$ for all $n$ and $k$.     Suppose the sequence of partitions associate to $(\mu_k)_k$ is eventually disjoint.  If  $\fin(\varphi)$ is tall, then it has a Borel selector. 
\end{teo}

\proof  We recall that the Schreier barrier is the following collection of finite subsets of $\N$:
\[
\mathcal{S}=\{s\in\fin: \; |s|=\min(s)+1\}.
\] 
By the  well known theorem of Nash-Williams  \cite{nash-williams68}, any coloring $c:\mathcal{S}\to 2$ has homogeneous sets. More precisely, for any $N\subseteq \N$ infinite, there is  $M\subseteq N$ infinite such that for some $i\in\{0,1\}$,  
$c(s)=i$ for every $s\in [M]^{<\omega}\cap \mathcal{S}$. As usual, we denote by $\hom(c)$ the collection of all $c$-homogeneous sets. Moreover, $\hom(c)$ admits a Borel selector (\cite[Corollary 3.8]{GrebikUzca2018}). We will find a coloring $c$ of $\mathcal{S}$ such that $\hom(c)\subseteq \fin(\varphi)$ and thus $\fin(\varphi)$ also has a Borel selector. 

Let $p\in \N$ as in \eqref{evenDis}. Consider the following coloring $\mathcal{S}$:
\[
c(\{q,n_1, \cdots, n_q\})=1\;\;\; \mbox{iff there is $k\in \N$ such that $\mu_k(\{n_j\})\geq 2^{-p-1}$ for all $1\leq j\leq q$.} 
\]
We will show that  $hom(c)\su \fin(\varphi)$.

We first show that every infinite homogeneous set is of color 0.
In fact, let $A$ be an infinite set. It suffices to find $s\su A$ with  $s\in \mathcal{S}$ of color 0.   Since $\fin(\varphi)$ is tall, we also assume that 
$\varphi (A)<\infty$. 
Let $q\in A$ be such that $q2^{-p-1}>\varphi(A)$. Let $q<n_1<\cdots<n_q$ in $A$, we claim that 
$c(\{q,n_1, \cdots, n_q\})=0$. In fact, suppose not, and let $k$ be  such that  $\mu_k(\{n_j\})\geq 2^{-p-1}$ for all $1\leq j\leq q$. Thus 
\[
\varphi(A)\geq \mu_k(\{n_1, \cdots, n_q\})=\sum_{j=1}^q\mu_k(\{n_j\})\geq q 2^{-p-1}>\varphi(A),
\]
a contradiction. 

To finish the proof, it suffices to show that every infinite homogeneous set $H$ of color 0 belongs to $\fin(\varphi)$. 
Recall the definition of the partition $\mathcal{P}_n$ associated to $(\mu_k)_k$:   $\mathcal{P}_n=\{A^n_i:\; i\in\N\cup\{\infty\}\}$ where each piece $A^n_i$ is defined by \eqref{partition}. 

Let $q=\min(H)$ and $F\su H\setminus\{q\}$ be a finite set.  For each $k\in \N$, let 
\[
F_k=\{n\in F:\;\mu_k(\{n\})\geq 2^{-p-1}\}.
\]
Fix $k\in \N$. Since $H$ is homogeneous of color $0$, $|F_k|<q$. Let  $n,m\in F\setminus F_k$ with $n\neq m$.  Then $\mu_k(\{n\})<2^{-p-1}$ and $\mu_k(\{m\})<2^{-p-1}$. Thus $k\in A^n_i\cap A^m_j$ for some $i,j> p$ and by \eqref{evenDis}, we have that $i\neq j$.  Thus
\[
\sum_{n\in F\setminus F_k}\mu_k(\{n\})\leq \sum_{i>p} 1/2^i.
\]
Then
 \[
\mu_k(F)=\sum_{n\in F}\mu_k(\{n\})= \sum_{n\in F_k}\mu_k(\{n\})+\sum_{n\in F\setminus F_k}\mu_k(\{n\}) \leq q+2.
\]
As this holds for every $k$, we have that $\varphi(H\setminus\{q\})\leq q+2$. Thus $H\in \fin(\varphi)$.
\endproof

\begin{prop}
\label{evenDis2} 
Let $\mathcal{Q}_n$, for each $n\in \N$,  be a collection  of   pairwise disjoint subsets of $\N$, say $\mathcal{Q}_n=\{B^n_i:\;i\in\N\}$. Let $L_n=\{i\in \N:\: B^n_i\neq \emptyset\}$. Suppose there is $l\in \N$ such  that $|L_n|\leq l$ for every $n$. 
Then,  there is an infinite set $A\su \N$ such that  $(\mathcal{Q}_n)_{n\in A}$  is eventually disjoint.  
\end{prop}

\proof By induction on $l$. Suppose $l=1$. For each $n$ let $i_n$  be such that  $\mathcal{Q}_n=\{B^n_{i_n}\}$. We consider two cases: (a)   There is $A\su \N$ infinite such that $(i_n)_{n\in A}$ is constant. Then  $(\mathcal{Q}_n)_{n\in A}$  is eventually disjoint. (b) There is $A\su \N$ such that $i\in A\mapsto i_n$ is 1-1, then  $(\mathcal{Q}_n)_{n\in A}$  is eventually disjoint. 

 Suppose it holds for any sequence of pairwise disjoint sets with at most $l$ non empty sets.  Let  $\mathcal{Q}_n=\{B^n_i:\;i\in\N\}$ be such that $|L_n|\leq l+1$.  Let $i_n=\min(L_n)$. 
There are two cases to be considered. (a) Suppose  $\sup_n i_n=\infty$. Pick $A=\{n_k:\; k\in \N\}$  such that $\max(L_{n_k})<\min(L_{n_{k+1}})$. Then $(\mathcal{Q}_{n_k})_{k\in \N}$ is eventually disjoint. 
(b)  There is $B\su \N$ infinite such that $(i_n)_{n\in B}$ is constant. Let $\mathcal{Q}_n'=\mathcal{Q}_n\setminus \{B^n_{i_n}\}$.  Then, we can apply the inductive hypothesis to $ (\mathcal{Q}_n')_{n\in B}$ and find $A\su B$ infinite such that $ (\mathcal{Q}_n')_{n\in A}$ is eventually disjoint. Then $ (\mathcal{Q}_n)_{n\in A}$ is also eventually disjoint.
\endproof

From the previous result and Theorem  \ref{evenDis3} we have 

\begin{coro}
Let $\varphi=\sup_k\mu_k$ where  each $\mu_k$ is a measure on $\N$ such that  $\mu_k(\{n\})\leq 1$ for all $n$ and $k$.   Let $\{A^n_i:\; i\in \N\}\cup \{A^n_\infty\}$ be the associated partitions and $L_n=\{i\in \N: A^n_i\neq \emptyset\}$.
Suppose there is $l$ such that $|L_n|\leq l$ for all $n$. If $\fin(\varphi)$ is tall, then it has a Borel selector.
\end{coro}

The main open questions we have left are the following:

\begin{question}
\label{nonpatho-selector}
Let $\ideal$ be  nonpathological $F_\sigma$ tall ideal. Does  $\mathcal{R}\leq_K \ideal$? Does $\ideal$ have a Borel selector?
\end{question}

\section{$\mathcal{B}$-representable ideals}
\label{B-representable}

In this section we show how to represent $F_{\sigma}$ ideals in a Banach space following the ideas introduced in  \cite{Borodulinetal2015,Drewnowski}.  We first start with the representation in Polish abelian groups and later in Banach spaces. 

Let $(G,+,d)$ be a Polish abelian group. We emphasis that the metric $d$ is part of the definition. In fact the ideal associated to the group depends on the metric used. Let ${\bf x}= (x_n)_n$ be a sequence in $G$. We say that $\sum_n x_n$ is 
{\em unconditionally convergent}, if there is $a\in G$ such that  $x_{\pi(0)}+x_{\pi(1)}+\cdots+ x_{\pi(n)} \to a$ for every permutation $\pi$ of $\N$.  We say that ${\bf x}$ is  {\em perfectly bounded}, if there exists $k>0$ such that for every $F\in \fin$ 
$$
d\left (0,\sum_{n\in F}x_n \right )\leq k.
$$
We introduce two ideals.  Given ${\bf x}= (x_n)_n$ a sequence in $G$, let 
$$
\mathcal{C}({\bf x})=\left \{ A\subseteq \mathbb{N}: \sum_{n\in A} x_n
 \text{ is unconditionally convergent}\right \}
 $$
and
$$
\mathcal{B}({\bf x})=\left \{ A\subseteq \mathbb{N}: \sum_{n\in A} x_n
 \text{ is perfectly bounded }\right \}.
 $$

We observe that the ideal $\B(x)$  depends on the metric of the group, not just on the topology, as is the case for $\C({\bf x})$.   

An ideal  $\ideal$  is  {\em Polish $\mathcal{B}$-representable}  (resp. {\em Polish $\C$-representable}) if there exists a Polish abelian group $G$ and  a sequence ${\bf x}=(x_n)_n$ in $G$ such that $\ideal=\mathcal{B}({\bf x})$ (resp. $\ideal=\C({\bf x})$). 
Polish $\C$-representable ideals were studied in \cite{Borodulinetal2015}.  As a consequence of Solecki's Theorem \ref{solecki} they proved the following.

\begin{teo}
\label{Polish C-repre}
 (Borodulin et al \cite{Borodulinetal2015}) An ideal $\ideal$ is Polish $\C$-representable iff it is an analytic $P$-ideal.
\end{teo}

\subsection{$\B$-representability in Polish groups}

Let $G$ be a Polish abelian group, $d$ a complete, translation invariant metric on $G$  and ${\bf x}=(x_n)_n$ a sequence in $G$.  We associate to ${\bf x}$ a lscsm $\varphi_{\bf x}$ as follows:  $\varphi_{\bf x}(\varnothing)=0$ and if $A\neq \varnothing$ we let
$$
\varphi_{\bf x}(A)=\sup \left \{ d\left (0,\sum_{n\in F}x_n \right ): \varnothing \neq F \in [A]^{< \omega} \right \}.
$$

We show that $\varphi_{\bf x}$ is indeed a lscsm.  From its definition, it is clear that $\varphi_{\bf x}(A)=\lim_{n\to\infty}\varphi_{\bf x}(A\cap\{0,1,\cdots, n\})$ for every $A\subseteq \N$. 
Let $A$ and $B$ be finite disjoint subsets of $\N$. Then, by the translation invariance of $d$, we have 

\[
d(0,\sum_{n\in A\cup B}x_n)=d(0, \sum_{n\in A}x_n+\sum_{n\in B}x_n)\leq d(0,\sum_{n\in A}x_n)+d(\sum_{n\in A}x_n,\sum_{n\in A}x_n+\sum_{n\in B}x_n)= d(0,\sum_{n\in A}x_n)+d(0,\sum_{n\in B}x_n).
\]

\medskip

Let $A,B$ be arbitrary subsets of $\N$,  and $\varepsilon>0$. Take a finite subset $F$ of $A\cup B$ such that $d(0,\sum_{n\in F}x_n)\geq \varphi_{\bf x}(A\cup B)-\varepsilon$.  
Since $d(0,\sum_{n\in F}x_n)\leq d(0,\sum_{n\in F\cap A}x_n)+d(0,\sum_{n\in F\setminus A}x_n)\leq \varphi_{\bf x}(A)+\varphi_{\bf x}(B)$,  it follows that $ \varphi_{\bf x}$ is subadditive and thus is a lscsm.

\begin{lema}
\label{B-F-sigma}
Let $G$ be a Polish abelian group and ${\bf x}=(x_n)_n$ a sequence in $G$. 
Then $\mathcal{B}({\bf x })=\fin(\varphi_{\bf x})$ and  $\B({\bf x})$ is $F_{\sigma}$.
\end{lema}

\begin{proof}
Let $A\in \mathcal{B}({\bf x})$. Then there exists $k>0$ such that for every $\varnothing \neq F\in [A]^{<\omega}$,
$$
d\left (0, \sum_{n\in F}x_n \right )\leq k.
$$
By the definition of $\varphi_{\bf x}$, we have $\varphi_{\bf x}(A)\leq k$. Hence $A\in \fin(\varphi_{\bf x})$.
Conversely, assume that $A\in \fin(\varphi_{\bf x})$, then there exists $k>0$ such that $\varphi_{\bf x}(A)\leq k$. By the definition of $\varphi_{\bf x}$, we clearly  have that $A\in \mathcal{B}({\bf x })$. Thus $\mathcal{B}({\bf x })=\fin(\varphi_{\bf x})$. Finally, since $\fin(\varphi_{\bf x})$ is $F_\sigma$ (see Theorem \ref{mazur}), so is $\B({\bf x})$.
 \end{proof}

\begin{teo} The following statements are equivalent for any ideal $\ideal$ on $\N$. 
\begin{itemize}
\item[(i)] $\ideal$ is $F_\sigma$.

\item[(ii)] $\ideal$  is $\B$-representable in $(\fin,d)$ for some  compatible metric $d$ on $\fin$ (as discrete topological group).

\item[(iii)] $\ideal$ is Polish $\B$-representable. 
 \end{itemize}
 \end{teo}
 
 \proof 
 By Lemma \ref{B-F-sigma}, (i) follows from (iii) and clearly (ii) implies (iii). To see that 
(i) implies (ii), let $\ideal$ be a $F_{\sigma}$ ideal on $\N$. By Theorem \ref{mazur}, there is a lscsm $\varphi$ taking  values on $\N\cup\{\infty\}$ such that $\ideal=\fin(\varphi)$. Then $\exh{(\varphi)}=\fin$.  From the proof of Solecki's theorem \ref{solecki} we know that the complete metric on $\fin$ given by $d(A,B)=\varphi (A\triangle B)$ is compatible with the group structure of $\fin$. 
Let $x_n=\{n\}$ and ${\bf x}=(x_n)_n$.  We claim that $\ideal=\B({\bf x})$ in the Polish group $(\fin, d)$. 
First note that for every $\emptyset\neq F\in \fin$, $ F=\sum_{n\in F}x_n$. 
Thus
\[
d\left (\varnothing, \sum_{n\in F}x_n \right )=\varphi(F). \label{Note33}
\]
By the lower semicontinuity of $\varphi$, we conclude that $\varphi=\varphi_{\bf x}$. Therefore, by Lemma \ref{B-F-sigma}, $\fin(\varphi)=\B({\bf x})$.

\qed
 
Notice that the proof of  previous result shows  that  what matter for  this type of representation is the translation invariant metric used on $\fin$, the topology is irrelevant as it can be assumed to be the discrete topology.

\subsection{$\B$-representability in Banach spaces}

The motivating example of Polish representability is when the group is a Banach space. We rephrase the definitions of $\C({\bf x})$ and $\B({\bf x})$ for the context of a Banach space.
Let  ${\bf x}=(x_n)_n$ be a sequence in  $X$. 

\begin{itemize}
\item $\sum x_n$ converges unconditionally, if $\sum x_{\pi(n)}$ converges for all permutation $\pi:\N \rightarrow \N$.
    
\item $\sum x_n$ is perfectly bounded, if there is  $k>0$ such that for all $F\subset \N$ finite, $
\left \| \sum_{n\in F}x_n \right \|\leq k$.

\item The lscsm associated to $\bf x$ is given by  $\varphi_{\bf x}(\emptyset)=0$ and for $A\su \N$ non empty, we put
\begin{equation}
\label{fisubx}
\varphi_{\bf x}(A)=\sup\{ \|\sum_{n\in F}x_n\|:\;  F\su A\; \mbox{is finite non empty}\}.
\end{equation}
\end{itemize}

A  motivation for studying $\B({\bf x})$ comes from the next result (part (iii) was not explicitly included but follows from the proof Theorem 1.3  of \cite{Drewnowski}). 
\begin{teo}\label{DreLabu} (Drewnowski-Labuda \cite{Drewnowski}) Let $X$ be a Banach space. The following statements are equivalent:
\begin{itemize}
    \item[(i)] $X$ does not contain an isomorphic copy of  $c_0$.
    \item[(ii)]  $\mathcal{C}({\bf{x}})$ is $F_\sigma$ in  $2^{\N}$ for each sequence ${\bf{x}}$ in  $X$. 
    \item[(iii)]  $\mathcal{C}({\bf{x}})=\mathcal{B}({\bf{x}})$ for each sequence ${\bf{x}}$ in $X$.
\end{itemize}
\end{teo}

When working in Banach spaces, Theorem \ref{Polish C-repre}
is strengthened as follows. 

\begin{teo}\label{C-repre}
 (Borodulin et al \cite{Borodulinetal2015}) Let $\ideal$ be an ideal on $\N$. The following are equivalent:
\begin{itemize}
\item[(i)] $\ideal=\exh(\varphi)$ for a nonpathological lscsm $\varphi$. 
\item[(ii)] $\ideal=\C({\bf x})$ for some sequence ${\bf x}=(x_n)_n$ in a Banach space.
\end{itemize}
\end{teo}

The proof of the previous result  also provides a characterization of $\B$-representability on  Banach spaces, as we show below.   Since $l_\infty$ contains isometric copies of all separable Banach spaces,  we have the following (already used in \cite{Borodulinetal2015} for $\C({\bf x})$). 

\begin{prop}
\label{linfinito}
Let  ${\bf x}=(x_n)_n$ be a sequence in a Banach space $X$. There is ${\bf y}=(y_n)_n$ in $l_\infty$ such that 
$\B({\bf x})=\B({\bf y})$. 
\end{prop}

From now on, we only work with $l_\infty$ (or $c_0$),  this assumption implies that  $\varphi_{\bf x}$ has  the  important extra feature of being nonpathological. 

It is convenient to have that the vectors  $x_n\in l_\infty$ used in the representation of an ideal are of non negative terms. 
The following result was proved in \cite{Borodulinetal2015} for $\C({\bf x})$, a similar argument works also for $\B({\bf x})$.
 
\begin{lema}
\label{Banach10}
Let ${\bf x}=(x_n)_n$ be a sequence in $l_\infty$. Let  ${\bf x'}=(x'_n)$, where  $x'_n(k)=|x_n(k)|$ for each  $n,k\in \N$, then  $\mathcal{B}({\bf x})=\mathcal{B}({\bf x'})$.
\end{lema}

\bigskip

Now we recall why the lscsm $\varphi_{\bf x}$ given by \eqref{fisubx} is nonpathological when working on  $l_\infty$. 
Let ${\bf{x}}=(x_n)_n$ be a sequence in $l_\infty$ and assume that   $x_n(k)\geq 0$ for all $n,k\in \N$. Define  a sequence of measures as follows.  For $A\subseteq \N$ and  $k\in \N$, let 
$$
\mu_k(A)=\sum_{n\in A}x_n(k).
$$
Let $ \psi=\sup_k \mu_k$, thus $\psi$ is a nonpathological lscsm.  Notice that $\psi(\{n\})= \|x_n\|_\infty$ for all $n$ and, more generally,    for each $F\su\N$ finite we have 
$$
\psi(F)=\left \| \sum_{n\in F}x_n \right \|_\infty.
$$
Since $\psi$ is monotone,  $\psi(F)=\varphi_{\bf{x}} (F)$ for every finite set $F$.  Therefore $\psi=\varphi_{\bf{x}}$.

\begin{teo} Let ${\bf{x}}=(x_n)_n$  be a sequence in $l_\infty$ with  $x_n\geq 0$ for all $n$. Then  $\varphi_{\bf{x}}$ is a nonpathological and 
\begin{itemize}
\item[(i)] $\mathcal{C}({\bf{x}})=\exh(\varphi_{\bf{x}})$. 
 
\item[(ii)] $\mathcal{B}({\bf{x}})=\fin(\varphi_{\bf{x}})$. 
\end{itemize}
 \end{teo}
\proof (i)  follows from the proof of Theorem 4.4 of \cite{Borodulinetal2015}.
(ii) Follows from Lemma \ref{B-F-sigma}.
\endproof

Conversely, given a nonpathological lscsm $\varphi$, say $\varphi=\sup_k \mu_k$, where  $\mu_k$ is a measure for each $k$, we associate to it a sequence ${\bf{x}}_\varphi=(x_n)_n $ of elements of $l_\infty$  as follows: Given  $n\in \N$, let
$$
x_n=(\mu_0(\{n\}),\ldots ,\mu_k(\{n\}),\ldots ).
$$
Notice that $\| x_n\|_\infty= \varphi(\{n\})$ for all $n$. 
For each $F\in \fin$, we have
$$
\varphi(F)=\sup\{\mu_k(F): \;k\in \N\}=\sup\{\sum_{n\in F}\mu_k(\{n\}): \;k\in \N\}=\sup\{\sum_{n\in F}x_n(k): \;k\in \N\}=\left \| \sum_{n\in F}x_n \right \|_\infty.
$$
In other words, $\varphi=\varphi_{{\bf x}_\varphi}$. 
Part (i) of the following result is  \cite[Theorem 4.4]{Borodulinetal2015} and (ii) follows from the above discussion. 

\begin{teo} Let $\varphi$ be a nonpathological lscsm. Then 
\begin{itemize}
\item[(i)] $\mathcal{C}({\bf{x}}_\varphi)=\exh(\varphi)$. 
 
\item[(ii)] $\mathcal{B}({\bf{x}}_\varphi)=\fin(\varphi)$. 
\end{itemize}
 \end{teo}

The following theorem is  analogous to Theorem \ref{C-repre} but for  $\B$-representability. 

\begin{teo}
\label{TeoremaB}
An ideal $\ideal$  is $\mathcal{B}$-representable in a Banach space if, and only if, there is a nonpathological lscsm  $\varphi$ such that $\ideal=\fin(\varphi)$.
\end{teo}

\begin{proof}
Suppose $\ideal$ is a $\mathcal{B}$-representable ideal in a Banach space. By Lemma \ref{linfinito}, we can assume that $\ideal$ is  $\mathcal{B}$-representable in $l_\infty$. Let ${\bf{x}}=(x_n)_n$ be a sequence in $l_\infty$ such that  $\ideal=\mathcal{B}({\bf{x}})$. By Lemma \ref{Banach10} we assume that $x_n\geq 0$ for all $n$.  Now the result follows from the previous considerations where we have shown that $\B({\bf x})=\fin(\varphi_{\bf x})$. Conversely, if $\varphi$ is nonpathological, we have shown above that $\fin(\varphi)=\B({\bf x}_\varphi)$. 
\end{proof}

To end this section, we present an example of an ideal which is  $\mathcal{B}$-representable in $c_0$ and is not a $P$-ideal, in particular, is not $\mathcal C$-representable in any Polish group.

\begin{ex}
\label{ejem-B-no-C}
{\em 
$\fin \times \{\varnothing \}$ is $\mathcal{B}$-representable on $c_0$. 
Recall that $\fin\times\{\emptyset\}$ is defined by  letting $A\in \fin \times \{\varnothing \}$ iff there is $k$ such that $A\subseteq B_0\cup\cdots\cup B_k$, where  $(B_n)_n$ is a partition of  $\N$ into infinitely many infinite sets. It is well known, and easy to verify, that $\fin \times \{\varnothing \}$ is not a $P$-ideal. Let  ${\bf{x}}=(x_n)_n$ be given by $x_n=me_n$, for $n\in B_m$, where  $(e_n)_n$ is the usual base for $c_0$.  It is easy to verify that  $\fin \times \{\varnothing \}=\mathcal{B}({\bf{x}})$. 
}
\end{ex}

\subsection{Tallness of  $\B({\bf x})$}
\label{tallnessB}
It is easy to check that  $\C({\bf x})$ is tall iff $\|x_n\|\to 0$. We show that the tallness  of $\B({\bf x})$ is related to the weak topology. 
A classical characterization of perfect boundedness  is as follows (see, for instance, \cite[Lemma 2.4.6]{albiac-kalton}). 
 
 \begin{prop}
 \label{wuc}
 A series $\sum_nx_n$ in a Banach space is perfectly bounded iff $\sum_n |x^*(x_n)|<\infty$ for all $x^*\in X^*$. 
\end{prop}

From this we get the following. 

\begin{prop}
\label{tall-wn}
Let ${\bf{x}}=(x_n)_n$ be a sequence in a Banach space $X$. If $\mathcal{B}({\bf x})$ is tall,  then  ${\bf x}=(x_n)_n$ is weakly null.
\end{prop}

\begin{proof}
Suppose $\B({\bf x})$ is tall and $(x_n)_n$ is not weakly null. Then there is $A\su \N$ infinite  and $x^*\in X^*$ such that $\inf_{n\in A}x^*(x_n)>0$. Let $B\su A$ infinite such that $\sum_{n\in B}x_n$ is perfectly bounded. This contradicts proposition \ref{wuc}.
\end{proof}

Thus, for a sequence ${\bf x}=(x_n)_n$, we have the following implications:
\[
(x_n)_n \; \mbox{is $\|\cdot \|$-null} \;\Rightarrow \B({\bf x})\; \mbox{is tall}\;\Rightarrow (x_n)_n\; \mbox{is weakly null}.
\]
This implications are, in general, strict. However, for $c_0$ the last one is an equivalence, as we show next.  

The following result  was originally proved using the classical Bessaga-Pelczy\'nski's selection theorem and it was the motivation for Theorem \ref{c0tall}.  

\begin{teo}
\label{c0-tall}
Let ${\bf x}=(x_n)_n$ be a sequence in $c_0$. Then $\mathcal{R}\leq_K \B({\bf x} )$ iff $(x_n)_n$ is weakly null.
\end{teo}

\proof
If  $\mathcal{R}\leq_K \B({\bf x} )$, then $\B({\bf x})$ is tall and thus ${\bf x}=(x_n)_n$ is weakly null by  Proposition \ref{tall-wn}. Conversely, let $(x_n)_n$ be a weakly null sequence in $c_0$. By Lemma \ref{Banach10}, we can assume that $x_n(k)\geq 0$ for all $n$ and $k$.  The corresponding measures are given by $\mu_k(\{n\})=x_n(k)$.  Since each $x_n\in c_0$,     $(\mu_k)_k$ is of type $c_0$ (see definition \ref{type-c0}). Notice that  $\varphi_{\bf x}=\sup_k\mu_k$ and  thus $\B({\bf x})=\fin(\varphi_{\bf x})$. Condition (iii) in Theorem \ref{c0tall} (namely, $\lim_n\mu_k(\{n\})=0$ for all $k$) is the translation of being weakly null in $c_0$. Thus $\mathcal{R}\leq_K\fin(\varphi_{\bf x})$. 
\endproof

The previous result naturally suggests the following.

\begin{question}
Which Banach spaces satisfies  that $\mathcal{R}\leq_K \B({\bf x} )$ for every weakly null sequence ${\bf x}=(x_n)_n$? 
\end{question}

In relation to the previous question. Let   ${\bf x}=(x_n)_n$ be the usual unit basis of $l_2$ (which is weakly null). Since $l_2$ does not contain copies of $c_0$,  $\B({\bf x})=\fin$.  Moreover, the same happens in $l_\infty$, as this space contains isomorphic copies of every separable Banach space. 

\bigskip

\noindent{\bf Acknowledgment.} We thank Michael Rinc\'on Villamizar for some comments about an earlier version of this  paper and for bringing to our attention the article \cite{Drewnowski}.
We  thank  J. Tryba for his comments about the construction of pathological ideals.  
Finally, we also thank the anonymous referee for their careful reading and remarks, which helped clarify some inaccuracies presented in the first version of this article.

\end{document}